\documentclass{amsart}
\usepackage[utf8]{inputenc}
\usepackage[T2A]{fontenc}
\usepackage{datetime}
\usepackage{color}
\usepackage{graphicx}

\newcounter{draft}
\ifnum\value{draft}>0
\usepackage{showkeys}
\usepackage{ulem}

\newcommand{\comment}[1]{#1}
\else
\newcommand{\comment}[1]{}
\fi

\newdateformat{monthyeardate}{%
  \monthname[\THEMONTH], \THEYEAR}

\usepackage{amsmath,amsthm,amssymb,amsopn}
\usepackage{stmaryrd}
\usepackage{mathabx}

\usepackage{algorithm}
\usepackage[noend]{algpseudocode}

\algnewcommand{\LeftComment}[1]{\Statex \(\triangleright\) #1} 

\theoremstyle{plain}
\newtheorem{theorem}{Theorem}[section]

\newtheorem{lemma}[theorem]{Lemma}
\newtheorem{proposition}[theorem]{Proposition}

\newtheorem{corollary}[theorem]{Corollary}
\theoremstyle{remark}
\newtheorem{remark}[theorem]{Remark}

\theoremstyle{definition}

\newcommand{\bbR}{\mathbb{R}}
\newcommand{\bbN}{\mathbb{N}}
\newcommand{\bbQ}{\mathbb{Q}}

\newcommand{\cP}{\mathcal{P}}
\newcommand{\cF}{\mathcal{F}}

\newcommand{\ot}{\mathsf{ot}}

\newcommand{\tameF}{\cF'_n}

\newcommand{\Succ}{\textsc{Succ}}
\newcommand{\BuiltSucc}{\textsc{BuiltSucc}}
\newcommand{\Pred}{\textsc{Pred}}
\newcommand{\WeakPred}{\textsc{WeakPred}}
\newcommand{\closure}{\mathsf{cl}}
\newcommand{\eps}{\varepsilon}

\begin{document}
\title{Generalized fusible numbers and their ordinals}
\author{Alexander~I.~Bufetov}
\address{ CNRS, Aix-Marseille Universit{\'e}, Centrale Marseille,  Institut de Math{\'e}matiques de Marseille, UMR7373,  39 Rue F. Joliot Curie 13453, Marseille, France;\newline
  Steklov  Mathematical Institute of RAS, Moscow, Russia;\newline
  Institute for Information Transmission Problems, Moscow, Russia.  }
  \email{alexander.bufetov@univ-amu.fr, bufetov@mi-ras.ru}

\author{Gabriel~Nivasch}
\address{Ariel University, Ariel, Israel.}
\email{ gabrieln@ariel.ac.il}

\author{Fedor~Pakhomov}
\address{
Ghent University, Ghent, Belgium;\newline
  Steklov  Mathematical Institute of RAS, Moscow, Russia.}

  \email{fedor.pakhomov@ugent.be, pakhfn@mi-ras.ru}

\maketitle

\begin{abstract}
  Erickson defined the \emph{fusible numbers} as a set $\cF$ of reals generated by repeated application of the function $\frac{x+y+1}{2}$. Erickson, Nivasch, and Xu showed that $\cF$ is well ordered, with order type $\eps_0$. They also investigated a recursively defined function $M\colon \mathbb{R}\to\mathbb{R}$. They showed that the set of points of discontinuity of $M$ is a subset of $\cF$ of order type $\eps_0$. They also showed that, although $M$ is a total function on $\bbR$, the fact that the restriction of $M$ to $\mathbb{Q}$ is total is not provable in first-order Peano arithmetic $\mathsf{PA}$.

In this paper we explore the problem (raised by Friedman) of whether similar approaches can yield well-ordered sets $\cF$ of larger order types. As Friedman pointed out, Kruskal's tree theorem yields an upper bound of the small Veblen ordinal for the order type of any set generated in a similar way by repeated application of a monotone function $g:\bbR^n\to\bbR$.

 The most straightforward generalization of $\frac{x+y+1}{2}$ to an $n$-ary function is the function $\frac{x_1+\cdots+x_n+1}{n}$. We show that this function generates a set $\cF_n$ whose order type is just $\varphi_{n-1}(0)$. For this, we develop recursively defined functions $M_n\colon \mathbb{R}\to\mathbb{R}$ naturally generalizing the function $M$.

Furthermore, we prove that for any \emph{linear} function $g:\bbR^n\to\bbR$, the order type of the resulting $\cF$ is at most $\varphi_{n-1}(0)$.

Finally, we show that there do exist continuous functions $g:\bbR^n\to\bbR$ for which the order types of the resulting sets $\cF$ approach the small Veblen ordinal.
\end{abstract}

\section{Introduction}\label{sec_intro}

Jeff Erickson \cite{Eri} defined the set $\cF\subset\bbR$ of \emph{fusible numbers} the following way: Let $g:\bbR^2\to\bbR$ be given by $g(x,y) = (x+y+1)/2$. Then let $\cF\subset\bbR$ be the inclusion-least set that satisfies $0\in\cF$ and $g(x,y)\in\cF$ whenever $x,y\in\cF$ and $g(x,y)>\max{\{x,y\}}$. In other words, $\cF$ is constructing iteratively by initially inserting $0$ into $\cF$, and inserting $g(x,y)$ into $\cF$ whenever $x$ and $y$ are previously constructed elements smaller than $g(x,y)$. Erickson claimed in presentation \cite{Eri} that  $\cF$ is well-ordered with order type $\eps_0$. Although, as pointed by Junyan Xu \cite{Xu12}, Erickson's sketch was faulty, the claim itself is correct: The fact that $\cF$ is a well-order of type $\ge \eps_0$ was proved by Xu \cite{Xu12}, and the fact that the order type of $\cF$ is exactly $\eps_0$ was proved by Erickson, Nivasch, and Xu \cite{ENX20}.

Also, the following recursive algorithm defining a function $M\colon \mathbb{R}\to\mathbb{R}$ goes back to Erickson's presentation \cite{Eri}: $$M(x)=\begin{cases} -x, & \text{if $x<0$}\\ \frac{M(x-M(x-1))}{2}, &\text{otherwise.}\end{cases}$$ The function $M$ was further researched in \cite{Xu12,ENX20}. It has been established that $M$ terminates on all real inputs, and that the set $\cF'=\{x+M(x)\mid x\in\mathbb{R}\}$ coincides with the set of all points of discontinuity of $M$ and is a subset of $\cF$ with order type $\eps_0$. Finally, in \cite{ENX20} it was established that although $M$ is a total function on $\mathbb{R}$, its restriction to  $\mathbb{Q}$ is a total computable function whose totality cannot be proved in the first-order Peano arithmetic $\mathsf{PA}$. 

Friedman \cite{Fri20} then pointed out the following: Let $G=\{g_1,\ldots,g_k\}$ be any finite set of monotone functions of the form $g_i:\bbR^{n_i}\to\bbR$ (meaning $g(x_1,\ldots,x_{n_i})\le g(y_1,\ldots,y_{n_i})$ whenever $x_j\le y_j$ for every $1\le j\le n_i$), and let $P\subset \bbR$ be any finite initial set. Let $\cF=\cF(G,P)\subset\bbR$ be the inclusion-least set that satisfies $P\subseteq\cF$ and $g_i(x_1,\ldots,x_{n_i})\in\cF$ whenever $x_1,\ldots,x_{n_i}\in\cF$ with $g_i(x_1,\ldots, x_{n_i})>\max{\{x_1,\ldots,x_{n_i}\}}$. Then Kruskal's tree theorem \cite{kruskal} implies that $\cF$ is well ordered. Hence, the order type of $\cF$ can be upper-bounded using known bounds on the order types of linearizations of certain well partial orders of trees \cite{RW93,Sch79,Sch20}. For example, if $G=\{g\}$ has a single function, and $g$ is $n$-ary, $n\ge 3$, then the order type of the resulting set $\cF$ is at most $\varphi(1,\underbrace{0,\ldots,0}\limits_{\mbox{\footnotesize \textrm{$n$ times}}})$, where $\varphi$ is the $(n+1)$-ary Veblen function.  The limit of $\varphi(1,\underbrace{0,\ldots,0}\limits_{\mbox{\footnotesize \textrm{$n$ times}}})$ as $n\to\infty$ is known as the \emph{small Veblen ordinal}. Hence, the question arises whether there exist natural functions $g$ for which the resulting sets $\cF(\{g\},\{0\})$ have order types approaching the small Veblen ordinal. There is one particularly well-known ordinal $\Gamma_0=\varphi(1,0,0)$  that lies  between $\varepsilon_0$ and the small Veblen ordinal (it is also called Feferman-Schütte ordinal and is known to be the proof theoretic ordinal of the system $\mathrm{ATR}_0$ of second-order arithmetic \cite{FMS82}), so in particular it would be interesting to find a natural function $g$ on reals giving rise to this ordinal. 

Perhaps the most straightforward generalization of the function $g_2(x,y)=(x+y+1)/2$ are the functions $g_n(x_1, \ldots, x_n) = (x_1 + \cdots + x_n+1)/n$. We note that the function $g_2$ and set $\cF=\cF_2=\cF(\{g_2\},\{0\})$ has a natural motivation in terms of fuses that goes back to Erickson \cite{Eri}. If $x,y$, $|x-y|<1$ are moments in time measured in minutes, and we are given a fuse that burns off completely in exactly one minute if ignited from one end, then $g_2(x,y)$ is the moment of time when the fuse burns off if it is ignited from one end at moment $x$ and from the other end at moment $y$. Then $\cF_2$ is the set of time durations that could be measured by a procedure involving only ignition of such fuses. This motivation for $g_2$ has a natural analogue for the case of other functions $g_n$. Suppose we have a finite collection of water tanks, where each water tank has $n$ faucets. Each faucet, on its own, can empty the water tank in $1$ hour. A faucet can only be opened at time $0$ or at the moment another tank has become completely empty. We are interested in the set of times $t$ for which there is a way to make a tank empty precisely at time $t$ according to these rules. Since a tank whose faucets are opened at times $x_1, \ldots, x_k$ becomes empty at time $(x_1 + \cdots x_k + 1)/k$, the corresponding set is $\cF_{\le n}=\cF(\{g_1, \ldots, g_n\},\{0\})$. If we require that all of a tank's faucets must be opened before the tank empties, then the set we obtain is $\cF_n=\cF(\{g_n\},\{0\})$. Thus we have natural questions to determine the order types of these sets $\cF_n$ and $\cF_{\le n}$. We call the elements of $\cF_n$ \emph{$n$-fusible numbers}.



\subsection{Our results}

In this paper we prove that the order type of the above-mentioned sets $\cF_n$ and $\cF_{\le n}$ is $\varphi_{n-1}(0)$. The lower bound is achieved by considering  variants of the algorithm $M$:

\begin{theorem}\label{thm_M}
Let $n\ge 2$ be fixed, and consider the algorithm $M_n$ given by:
\begin{equation*}
    M_n(x) = \begin{cases} -x, &\text{if $x<0$};\\ \underbrace{M_n(x-M_n(x-M_n(\cdots M_n(x-1)\cdots)))}_\text{$n$ instances of $M_n$}/n, &\text{otherwise.}\end{cases}
\end{equation*}
Then $M_n$ terminates on all real inputs. Furthermore, $\cF'_n=\{x+M_n(x)\mid x\in \mathbb{R}\}$ is a subset of $\cF_n$ with order type $\varphi_{n-1}(0)$.
\end{theorem}

We prove that $\varphi_{n-1}(0)$ is not only a lower bound for the order type of $\cF_n$, but also an upper bound. Furthermore $\varphi_{n-1}(0)$ is an upper bound on the order type of any set generated by one linear $n$-ary function.

\begin{theorem}\label{thm_upper}
Let $g:\bbR^n\to \bbR$ be a monotone linear function, and let $P\subset\bbR$ be finite. Then the order type of $\cF(\{g\},P)$ is at most $\varphi_{n-1}(0)$.
\end{theorem}

This upper bound is somewhat disappointing, since it was reasonable to expect to get order types approaching the small Veblen ordinal.

We prove that as long as we consider arbitrary continuous functions, we get the optimal ordinals prescribed by the bounds for Kruskal tree theorem (for the bounds see \cite{Sch79,Sch20}):

\begin{theorem}\label{thm_cont}
For each $n\ge 3$ there exists a continuous monotone function $g:\bbR^n\to\bbR$ such that $\cF=\cF(\{g\},\{0\})$ has order type $\varphi(1,\underbrace{0,\ldots,0}\limits_{\mbox{\footnotesize \textrm{$n$ times}}})$.
\end{theorem}

The function $g$ of Theorem \ref{thm_cont}, however, is quite artificial, since is built ``backwards''. Namely we first build a certain monotone bijection on ordinals $\bar{\varphi}^\star\colon\Lambda^n\to\Lambda\setminus\{0\}$, where $\Lambda=\varphi(1,\underbrace{0,\ldots,0}\limits_{\mbox{\footnotesize \textrm{$n$ times}}})$. Then we embed $\Lambda$ into the reals, and finally we find a continuous $g$ whose restriction to the image of $\Lambda$ agrees with $\bar{\varphi}^\star$.

Finally, we make some progress on the computational problem of deciding whether a given $z\in\bbR$ belongs to $\cF(G,P)$ or not:

 \begin{theorem}\label{thm_alg}
Let $g:\bbR^n\to\bbR$ be a monotone linear function, let $P$ be a finite set, and let $\cF=\cF(\{g\},P)$. Then there exists an algorithm deciding whether a given $x\in\bbR$ belongs to the closure $\overline{\cF}$ or not.
\end{theorem}

The question of whether there exists an algorithm that decides whether $x\in\cF$ or not is still open.

This paper is organized as follows: Section \ref{sec_basic} introduces notation and presents some basic properties of the sets $\cF$. Section \ref{sec_backg} gives some background on ordinals below the small Veblen ordinal. Section \ref{sec_WPO} gives background on well partial orders. Sections \ref{Upper_bound_section}, \ref{sec_M}, \ref{sec_cont}, \ref{sec_alg} contain the proof of Theorems \ref{thm_upper}, \ref{thm_M}, \ref{thm_cont}, \ref{thm_alg} respectively. Section \ref{sec_closure} contains a result needed for Section \ref{sec_alg}.

\section{Notation and basic properties}\label{sec_basic}

Throughout this paper, given a set $G$ of monotone functions  $g:\bbR^{k_g}\to\bbR$ and a set $P\subseteq \mathbb{R}$, we denote as $\cF(G,P)=\cF$ the inclusion-least set such that $P\subseteq \cF(G,P)$ and for any $g\in G$, $x_1,\ldots,x_{k_g}\in \cF$ we have
\begin{center} $g(x_1, \ldots, x_{k_g})>\max\limits_{1\le i \le k_g}x_i$ \quad$\Rightarrow$\quad $g(x_1,\ldots,x_{k_g})\in\cF$.
\end{center}

In other words $\cF(G,P)$ is the set of values of all closed terms $t$ built from constants from $P$ and functions from $G$, such that for any subterm $g(u_1,\ldots,u_{k_g})$ of $t$, the value of $g(u_1,\ldots,u_{k_g})$ is greater than the values of $u_1,\ldots,u_{k_g}$. We call such terms \emph{monotone} terms.

For the specific functions $g_n:\bbR^n\to\bbR$ given by $g_n(x_1, \ldots, x_n) = (x_1 + \cdots + x_n + 1)/n$, we let  $\cF_n=\cF(\{g_n\},\{0\})$ and $\cF_{\le n}=\cF(\{g_1, \ldots, g_n\},\{0\})$.

We are primarily interested in cases where $G$ is finite. However, in our arguments we will also have to consider cases in which $G$ is infinite.

Friedman noted that the following proposition can be proved using Kruskal's theorem.
\begin{proposition}\label{prop_wo}
  Suppose  $G$ is a finite set of functions and $P$ is a well-ordered set of constants. Then $\cF(G,P)$ is well-ordered.
\end{proposition}

Below is a more direct proof of Proposition \ref{prop_wo} that uses only the infinite Ramsey theorem and completeness of the reals.
\begin{proof}
  We split the set $\cF(G,P)$ into the well-founded part $W$ and the ill-founded part $I$: $a\in\cF(G,P)$ is in $I$ if there is an infinite descending sequence starting from $a$, otherwise it is in $W$. For a contradiction assume that $I$ is not empty.

  Consider $\ell=\inf I$; note that it exists since any $a\in \cF(G,P)$ is bounded from below by $\inf P$. Clearly $I=\cF(G,P)\cap (\ell,+\infty)$ and $W=\cF(G,P)\cap (-\infty,\ell]$. For each $g\in G$, let $I_g\subseteq I$ be the set of all $a\in I$ that are of the form $g(\vec{p})$, $\vec{p}\in  W^k$, where $k$ is the arity of $g$. For any $a\in I$ there is a monotone term $t$ whose value is $a$, and there is a shortest subterm $t'$ of $t$ such that its value $a'$ lies in $I$. We observe that $a'\le a$ and either $a'\in P$ or $a'\in I_g$ for some $g\in G$. Thus $\inf (\bigcup\limits_{g\in G}I_g\cup (P\cap I)))=\ell$. Since $P$ is well-founded, either $P\cap I$ is empty or $\inf (P\cap I)>\ell$. Hence for some $g\in G$ we have $\inf I_g=\ell$. Further we fix $g$ with this property and denote its arity as $k$.

  We consider a strictly decreasing sequence $x_1>x_2>\ldots$ of elements of $I_g$ converging to $\ell$. For all $n$ we choose $y_{n,1},\ldots,y_{n,k}\in W$ such that $g(y_{n,1},\ldots,y_{n,k})=x_n$. Observe that by the infinite Ramsey Theorem there exists an infinite set $A\subseteq\bbN$ such that for each $1\le i\le n$ the sequence $\langle y_{n,i}\mid n\in A\rangle$ is either increasing, or constant, or decreasing. Indeed, we consider the coloring of increasing pairs of naturals $h\colon [\mathbb{N}]^2\to \{<,=,>\}^k$, where for each $n<m$ the value $h((n,m))$ is the tuple $(c_1,\ldots,c_k)$, where each $c_i$ is the result of comparing $y_{n,i}$ with $y_{n,j}$. Clearly, any infinite monochromatic set for this coloring indeed can serve as the desired $A$.

  However none of the sequences $\langle y_{n,i}\mid n\in A\rangle$ can be decreasing, since this would contradict the assumption that all $y_{n,i}\in W$. Hence by the monotonicity of $g$, the sequence $\langle g(y_{n,1},\ldots,y_{n,k})\mid n\in A\rangle$ is non-decreasing. But it is a subsequence of our decreasing sequence, contradiction.\end{proof}

\section{Background on ordinals}\label{sec_backg}

\subsection{Ordinals below $\varphi_\omega(0)$}

A function $f$ from ordinals to ordinals is called \emph{normal} if it is strictly monotone (meaning $\alpha<\beta$ implies $f(\alpha)<f(\beta)$) and continuous (meaning for every limit ordinal $\lambda$ we have $f(\alpha) = \lim_{\alpha<\lambda} f(\alpha)$). Given a normal function $f$ the \emph{derivative} $f'$ of $f$ is the normal function enumerating all the fixed points of $f$ in the increasing order. Alternatively $f'$ can be defined by transfinite recursion as
\begin{enumerate}
\item $f'(0) = \sup_{k\in\bbN} f^k(0)$;
\item $f'(\alpha+1) = \sup_{k\in\bbN} f^k(f'(\alpha)+1)$;
\item $f'(\lambda) = \sup_{\alpha<\lambda} f'(\alpha)$, for limit ordinals $\lambda$.
\end{enumerate}

The \emph{Veblen functions} $\varphi_n$, $n\in\bbN$ are a sequence of normal functions defined by starting with $\varphi_0(\alpha) = \omega^\alpha$, and for each $n\in\bbN$, letting $\varphi_{n+1}=\varphi_n'$. These functions can be defined more succintly by letting $\varphi_n(\alpha)$ be the least ordinal $\beta$ of the form $\omega^\gamma$ such that $\beta>\varphi_n(\alpha')$ for all $\alpha'<\alpha$, and such that $\beta>\varphi_{m}(\alpha')$ for all $m<n$ and $\alpha'<\beta$.

Denote $\varphi_{\omega}(0) = \lim_{n\in\bbN}\varphi_n(0)$. Then $\varphi_\omega(0)$ is the smallest ordinal that cannot be constructed by starting from $0$ and repeatedly applying ordinal addition and the functions $\varphi_n$, $n\in\bbN$ on previously constructed ordinals. The ordinals below $\varphi_\omega(0)$ are used in Sections \ref{Upper_bound_section} and \ref{sec_M} of this paper.

\subsection{The $n$-ary Veblen functions}\label{subsec_n-ary}

Fix $n\ge 3$. The $n$-ary Veblen function $\varphi\colon \mathsf{On}^n\to\mathsf{On}$ is a generalization of the Veblen function described above, with $\varphi_m(\alpha)$ now denoted $\varphi(0,\ldots,0,m,\alpha)$. The $n$-ary function $\varphi$ is defined by ordinal induction, by letting the function $\beta\mapsto \varphi(\alpha_1,\ldots,\alpha_k,\underbrace{0,\ldots,0}\limits_{\mbox{\footnotesize $n-k-1$}}, \beta)$, where $\alpha_k>0$, enumerate the common fixed points of the functions
\begin{equation*}
\bigl\{\gamma\mapsto\varphi(\alpha_1,\ldots,\alpha_{k-1},\alpha'_k,\gamma,\underbrace{0,\ldots,0}\limits_{\mbox{\footnotesize $n-k-1$}}):\alpha'_k<\alpha_k\bigr\}.
\end{equation*}
The following proposition can serve as an equivalent definition of $\varphi$:

\begin{proposition}\label{prop_veblen_equiv}
$\varphi(\alpha_1,\ldots,\alpha_n)$ is the least ordinal $\beta$ of the form $\omega^\gamma$ such that 
$$\varphi(\alpha_1,\ldots,\alpha_{k-1},\alpha'_k,\alpha'_{k+1},\ldots,\alpha'_{n})<\beta,$$ for all $1\le k\le n$, $\alpha'_k<\alpha_k$, and $\alpha'_{k+1},\ldots,\alpha'_n<\beta$.
\end{proposition}

In particular, the Feferman--Sch\"utte ordinal is $\Gamma_0 = \varphi(0,\ldots,0,1,0,0) = \varphi(1,0,0)$ (since leading zeros can be ignored). The limit of $\varphi(1,\underbrace{0,\ldots,0}\limits_{\mbox{\footnotesize $n$}})$ as $n\to\infty$ is known as the \emph{small Veblen ordinal}, and is the smallest ordinal that cannot be constructed by starting from $0$ and repeatedly applying Veblen functions of any finite arity to the previously constructed ordinals.

\subsection{Removing limit elements}

We will use several times the following simple observation:

\begin{proposition}\label{prop_remove_limits}
Let $S$ be an infinite well-ordered set, and let $S'$ be obtained from $S$ by removing the elements at limit positions. Then $\ot(S)=\ot(S')$ if $\ot(S)$ is a limit ordinal; otherwise $\ot(S)=\ot(S')+1$.
\end{proposition}

\section{Well partial orders}\label{sec_WPO}

A partial order $X=(D,\preceq)$ is said to be a \emph{well partial order} (wpo) if for every infinite sequence $a_0,a_1,a_2,\ldots$ of elements of $X$ there are naturals $i<j$ such that $a_i\preceq a_j$. A \emph{linearization} of a partial order $\preceq$ is a linear order $\le$ on the same domain $D$ such that $a\preceq b$ implies $a \le b$ for all $a,b\in X$. A partial order is a wpo iff all its linearizations are well-orders \cite{higman}. 
For a wpo set $X$, the expression $o(X)$ denotes the supremum of the order types of all the linearizations of $X$. A classical theorem of de Jongh and Parikh \cite{JP77} states that for every well-partially ordered set $X$ there exists a linearization $\leq$ whose order type is precisely $o(X)$.

There are a few standard constructions of new wpo's from given ones, which we will use in this paper. For wpo's $X$ and $Y$, both their disjoint union $X\sqcup Y$ and their product $X\times Y$ are wpo's.

Recall that $\oplus$ is the \emph{natural addition} of ordinals and $\otimes$ is the \emph{natural product} of ordinals, which are defined as follows: Given ordinals $\alpha,\beta$ with Cantor normal forms 
\begin{align*}
\alpha &=\omega^{\alpha_1}+\ldots+\omega^{\alpha_n},\qquad\text{ with $\alpha_1\ge\ldots\ge \alpha_n$},\\
\beta &=\omega^{\beta_1}+\ldots+\omega^{\beta_m},\qquad\text{ with $\beta_1\ge\ldots\ge\beta_m$};
\end{align*}
their natural sum $\alpha\oplus \beta$ is the ordinal $\gamma$ with the Cantor normal form $\omega^{\gamma_1}+\ldots+\omega^{\gamma_{n+m}}$, where $\gamma_1,\ldots,\gamma_{n+m}$ are $\alpha_1,\ldots,\alpha_n,\beta_1,\ldots,\beta_m$ sorted in nonincreasing order.  The natural product of $\alpha,\beta$ is given by $$\alpha\otimes \beta=\mathop{\bigoplus\limits_{1\le i\le n}}\limits_{1\le j\le m}\omega^{\alpha_i\oplus \beta_j}.$$

De Jongh and Parikh \cite{JP77} proved that $o(X\sqcup Y)=o(X)\oplus o(Y)$ and $o(X\times Y)=o(X)\otimes o(Y)$.

For a partially ordered set $X$, denote by $X^\star$ the set of all finite sequences of elements of $X$, partially ordered by $(a_1,...,a_n)\preceq (b_1,\ldots,b_m)$ if and only if there is a strictly monotone function $f\colon \{1,\ldots,n\}\to\{1,\ldots,m\}$ such that $a_i\preceq b_{f(i)}$, for all $1\le i\le n$. Higman's Lemma \cite{higman} states that if $X$ is wpo, then $X^{\star}$ is wpo as well. Schmidt \cite{Sch79} (see also \cite{Sch20}) proved that $o(W^{\star})\le \omega^{\omega^{o(W)+1}}$.

There is a variant of Kruskal's tree theorem that is relevant for this paper, where instead of trees the order is defined on terms. Suppose we are given sets $W_0,\ldots,W_n$ of function symbols, where the functions symbols in each $W_i$ are $i$-ary, and where each $W_i$ is endowed with a partial order $\preceq_{W_i}$. Let  us define the partial order $T(W_0,\ldots,W_n)$. The domain of the order is the least set of terms $K$ such that
\begin{center}$t_1,\ldots,t_k\in K$ and $a\in W_k$, where $0\le k\le n$ \quad$\Rightarrow$\quad $a(t_1,\ldots,t_k)\in K$.\end{center}
The comparison relation $\preceq_T$ is defined recursively as follows: Let $x,y\in K$ be two terms, where $x=a(t_1,\ldots,t_k)$, $a\in W_k$, $t_1, \ldots, t_k\in K$, and $y=b(u_1,\ldots,u_\ell)$, $b\in W_\ell$, $u_1,\ldots,u_\ell\in K$. Then we let $x\preceq_K y$ if $x\preceq_K u_i$ for some $1\le i\le\ell$, or $k=\ell$ and $a\preceq_{W_k} b$ and $t_i \preceq_K u_i$ for every $1\le i\le k$.

The variant of Kruskal's theorem that we need states that for any wpo's $W_0,\ldots,\allowbreak W_n$ the order $T(W_0,\ldots,W_n)$ is a wpo. Schmidt \cite{Sch79,Sch20} extensively studied the bounds for $o(T(W_0,\ldots,W_n))$. She mostly did it in terms of labeled trees, but also she discusses the term order \cite[Section~4]{Sch20}.

This version of Kruskal theorem allows an easy alternative proof of Proposition \ref{prop_wo}; we note that this observation is already contained in an e-mail by Friedman \cite{Fri20}. Consider finite set of monotone functions on reals $G$ and a well-ordered set of constants $P$. We split $G$ into sets $G_0,\ldots,G_n$, where each $G_i$ consists only of functions of the arity $i$. We endow $G_0,\ldots,G_n$ with the discrete orders and $P$ with the standard orders on reals. Clearly, the set of monotone terms for the pair $G,P$ is a subset of $T(P\sqcup G_0,G_1,\ldots,G_n)$ and for any  monotone terms $t,u$ if $t\preceq_T u$, then the value of $t$ is smaller than or equal to the value of $u$. Thus $\cF(G,P)$ is a lineariazation of a suborder of $T(P\sqcup G_0,G_1,\ldots,G_n)$ and thus is well-ordered. The construction above also shows that $\ot(\cF(G,P))\le o(T(P\sqcup G_0,G_1,\ldots,G_n))$.

\section{Monotone functions on reals that lead to ordinals below $\varphi_\omega(0)$}\label{Upper_bound_section}

In this section we prove Theorem \ref{thm_upper}.

For a monotone function $g\colon \bbR\to \bbR$ we denote as $A(g)\subseteq \bbR\cup\{+\infty\}$ the set of limit points $$\{\lim\limits_{n\to \infty} g^n(x) \mid x\in\bbR \text{ and }g(x)> x\}.$$
For a family of monotone functions $\langle h_{\vec{p}}\colon \bbR\to\bbR \mid \vec{p}\in \bbR^n\rangle$ and a set $P\subseteq \bbR$ we put $$A(h;P)=\bigcup\limits_{\vec{p}\in P^n} A(h_{\vec{p}}).$$

For a monotone function $g\colon \bbR^n\to\bbR$, $u\subseteq \{1,\ldots,n\}$, and $\vec{p}=(p_1,\ldots,p_{n-|u|})\in\bbR^{n-|u|}$ we denote as $g_{u,\vec{p}}\colon\bbR^{|u|}\to\bbR$ the result of substitution of parameters $\vec{p}$ instead of the arguments with indices from $\{1,\ldots,n\}\setminus u$. Formally the function $g_{u,\vec{p}}$ is the function that maps $(x_1,\ldots,x_{|u|})$ to $g(y_1,\ldots,y_n)$, where for each $i$-th element $j$ of $u$ the number $y_j$ is equal to $x_i$ and for each $i$-th element $j$ of $\{1,\ldots,n\}\setminus u$ the number $y_j$ is equal to $p_i$. We denote as $\dot g_{u,\vec{p}}\colon \bbR\to \bbR$ the function $$\dot g_{u,\vec{p}}(x)=g_{u,\vec{p}}(\underbrace{x,\ldots,x}\limits_{\mbox{\scriptsize $|u|$ times}}).$$ 
Hence, $A(\dot g_u,P)$ is the set of limit points of the functions $\dot g_{u,\vec p}$ for $\vec p\in P^{n-|u|}$.

We call a monotone function $g\colon \bbR^n\to \bbR$ \emph{tame} if there is $k$ such that for any $u\subsetneq\{1,\ldots,n\}$, any well-ordered set $P\subseteq \bbR$, and any well-ordered subset $B\subseteq A(\dot g_u,P)$, we have
$$\ot(B)<  \varphi_0^k({\ot(P)+1}).$$
As we will show below, linear functions are tame.


\begin{theorem}\label{tame_bound}
Suppose $g\colon \bbR^n\to \bbR$, $n\ge 2$ is a monotone tame function. Then $\ot(\cF(\{g\},\{0\}))\le \varphi_{n-1}(0)$.
\end{theorem}
\begin{proof}
  Let the function $g$ be fixed. For $m\le n$ and $P\subseteq \bbR$ we denote as $G_m(P)$ the set of multi-variate monotone functions on $\bbR$ consisting of all functions $g_{u,\vec{p}}$, where $u\subseteq \{1,\ldots,n\}$ is of  cardinality $m$ and $\vec{p}\in P^{n-|u|}$. We put $G_{\le m}(P)=\bigcup\limits_{m'\le m}G_{m'}(P)$, and $\cF_{\le m}(P) = \cF(G_{\le m}(P),P)$. Hence, $\cF_{\le m}(P)\subseteq \cF(\{g\},P)$ is the set of reals that can be constructed by starting from $P$ and repeatedly applying $g$ on at least $n-m$ elements of $P$ and at most $m$ other previously constructed elements.

As was noted in Section \ref{sec_WPO}, it follows from Kruskal's theorem that for any $0\le m\le n$ and any well-ordered set $P$, the set $\cF_{\le m}(P)$ is well-ordered.

For $1\le m< n$ we will prove by induction on $m$ that there is a natural number $k_{m}$ such that for any well-ordered $P$ we have  \begin{equation}\label{tmeq3}\ot(\cF_{\le m}(P))\le \varphi_{m-1}^{k_m}(\ot(P)+1).\end{equation}

First consider the case of $m=1$. Consider the wpo $W=\bigsqcup\limits_{1\le i\le n}(P,<)^{n-1}$. 
We put in correspondence to each unary function $h=g_{\{i\},\vec{p}}\in G_{1}(P)$ the element $e(h)$ of $W$ that is the tuple $\vec{p}$ from the $i$-th copy of $(P,<)^{n-1}$. Observe that if $e(h_1)\preceq_W e(h_2)$ then $h_1$ is pointwise less than or equal to $h_2$. Now considering the definition of $\cF_{\le 1}(P)=\cF(G_1(P), P\cup G_0(P))$ as the set of values of monotone terms, we observe that it is isomorphic to a linearization of a subset of $T((G_{0}(P),<)\sqcup (P,<),W)$. Hence $\ot(\cF_{\le 1}(P))\le o(T((G_{0}(P),<)\sqcup (P,<),W))$. Notice that in general $T(W_0,W_1)$ is isomorphic to $W_1^{\star}\times W_0$. Thus $$\ot(\cF_{\le 1}(P))\le o(((G_{0}(P),<)\sqcup (P,<))\times W^{\star}).$$

Now let us apply results about upper bounds for the order types of linearizations of wpo's. Observe that the set $G_0(P)$ is the set of reals consisting of values of the form $g(\vec{p})$, where $\vec{p}\in P^n$. Thus $(G_0(P),<)$ is isomorphic to a linearization of a subset of $(P,<)^n$. Hence  $$\ot(G_{0}(P)\sqcup P))\le o((P,<)\sqcup(P,<)^n)<(\ot(P,<)+1)^\omega.$$  In the same manner it is easy to see that $o(W)<(\ot(P)+1)^\omega$. Notice that
$$\begin{aligned}(\ot(P)+1)^\omega&\le (\varphi_0(\ot(P)+1))^\omega=\varphi_0((\ot(P)+1)\omega)\le \varphi_0(\varphi_0(\ot(P)+1)\varphi_0(1))\\ & =\varphi_0(\varphi_0((\ot(P)+1)+1))\le \varphi_0(\varphi_0(\varphi_0(\ot(P)+1)))=\varphi_0^3(\ot(P)+1).\end{aligned}$$
By the upper bound for Higman's Lemma we know that $$o(W^{\star})\le \varphi_0^2(o(W)+1)< \varphi_0^2(\varphi_0^3(\ot(P)+1))=\varphi_0^5(\ot(P)+1).$$
Thus 
$$\begin{aligned} \ot(\cF_{\le 1}(P))& \le o(((G_{0,P},<)\sqcup (P,<))\times W^{\star})\le \\ & \varphi_0^3(\ot(P)+1) \otimes \varphi_0^5(\ot(P)+1)\\ & = \varphi_0(\varphi_0^4(\ot(P)+1)+\varphi_0(\ot(P)+1))\le \varphi_0^6(\ot(P)+1).\end{aligned}$$

Now we assume that $m>1$. Let us consider the set of limit points $$A=\bigcup\limits_{u\in\cP^{=m}(\{1,\ldots,n\})} A(g_u,P)$$ and its subset $$\begin{aligned} A'=\{a\in A\mid  a=\lim_{i\to +\infty}\dot g^i_{u,\vec{p}}(b) &\text{, for some } b\in\cF_{\le m}(P)\text{, }u\in\mathcal{P}^{=m}(\{1,\ldots,n\}),\\ & \text{and }\vec{p}\in P^{n-m}\text{ such that }\dot g_{u,\vec{p}}(b)>\max(\vec p,b)\}.\end{aligned}$$ Since $\bigsqcup\limits_{u\in\cP^{=m}(\{1,\ldots,n\}}\cF_{\le m}(P)\times P^{n-m}$ is a wpo, the set $A'$ is well-ordered. Using the tameness of $g$, we conclude that $\ot(A')<\varphi_0^k(\ot(P)+1)\binom{n}{m}$, where $k$ is the tameness constant of $g$, and $\binom{n}{m}\in\bbN$ is a binomial coefficient. We further consider the set $A''=A'\cup\{\pm\infty\}$.

Let $\Lambda=\ot(A'')$ and $\langle a_\alpha\mid \alpha<\Lambda\rangle $ be enumeration of elements of $A''$ in increasing order. Our goal will be to show by transfinite induction on  $\alpha<\Lambda$ that \begin{equation}\label{tmeq}\ot(\cF_{\le m}(P)\cap (-\infty,a_\alpha))\le\varphi_{m-1}(\ot(P)+\alpha+1).\end{equation}



The case $\alpha=0$ follows since $a_0=-\infty$, so suppose $\alpha>0$. We put $b_0=\sup\limits_{\beta<\alpha}a_\beta$ and define $b_{i+1}\in\bbR\cup\{+\infty\}$ recursively as follows:
$$\begin{aligned}b_{i+1}=\min \{g_{v,\vec{r}}(\vec{d})\mid\,& v\in\cP^{=m}(\{1,\ldots,n\}),\; \vec{r}\in P^{n-m},\\ &  \vec{d}\in (\cF_{\le m}(P)\cap(b_i,\infty))^m\text{, and } g_{v,\vec{r}}(\vec{d})>\max(\vec{d},\vec{r}).\}\end{aligned}$$
Since $b_{i+1}$ is defined as the minimum of a subset of $\cF_{\le m}(P)$ it could be undefined only if the corresponding set is empty; in this case we put $b_{i+1}=+\infty$. 

We claim that $\lim_{i\to+\infty}b_i\ge a_{\alpha}$. If some $b_i$ is equal to $+\infty$, then the claim is trivially true. Hence we assume without loss of generality that all $b_i$ are below $+\infty$. For each $i$ we fix $v_i\in \cP^{=n-m}(\{1,\ldots,n\})$, $\vec{r}_i\in P^{n-m}$, and $\vec{d}_i\in (\cF_{\le m}(P)\cap(b_i,\infty))^m$ such that $b_{i+1}=g_{v_i,\vec{r}_i}(\vec{d}_i)$ and $\max(\vec{d}_i)<b_{i+1}$. We fix an infinite set of naturals $I_0$ such that there is a fixed $v$ such that $v_i=v$, for all $i\in I_0$. Since $P^{n-m}$ is a wpo, there is an infinite set of naturals $I_1\subset I_0$ such that the sequence $\langle \vec{r}_i \mid i\in I_1\rangle$ is either pointwise monotone increasing or constant.

Let $k=k_0<k_1<k_2<\ldots$ be the sequence enumerating the set $I_1$. By induction on $i$ we prove that $\dot g_{v,\vec{r}_{k}}^i(\max(\vec{d}_k))\le b_{k_i+1}$. Indeed, the base of the induction holds since $\max(\vec{d}_k)< b_{k+1}=b_{k_0+1}$ and the induction step follows since
\begin{multline*}\dot g_{v,\vec{r}_{k}}^{i+1}(\max(\vec{d}_k))\le \dot g_{v,\vec{r}_k}(b_{k_{i}+1})\le \dot g_{v,\vec{r}_k}(b_{k_{i+1}})\le \dot g_{v,\vec{r}_{k_{i+1}}}(b_{k_{i+1}})\\ \le g_{v,\vec{r}_{k_{i+1}}}(\vec{d}_{k_{i+1}})=b_{k_{i+1}+1}.
\end{multline*}
Thus $c=\lim_{i\to \infty} \dot g_{v,\vec{r}_k}^i(\max(\vec{d}_k))\le \lim_{i\to\infty} b_i$. Since  $\dot g_{v,\vec{r}_k}(\max(\vec{d}_k))>\max(\vec{d}_k)$, we have $c\in A$. Furthermore, since $\max(\vec{d}_k)\in \cF(G_{\le m}(P),P)$, we have $c\in A'$. Therefore, since $a_\alpha$ is the least element of $(A'\cup \{+\infty\})\cap [b_0,+\infty]$ and $c>b_0$, we have $a_\alpha\le c\le \lim_{i\to\infty} b_i$. This concludes the proof that $\lim_{i\to +\infty} b_i\ge a_\alpha$.

Let $U_{i}=\cF_{\le m}(P)\cap (-\infty,b_{i})$. Since $U_0=\cF_{\le m}(P)\cap (-\infty,\sup\limits_{\beta<\alpha}a_\beta)$, from the transfinite induction hypothesis (\ref{tmeq}) it follows that \begin{multline*}
\ot(U_0)=\sup\limits_{\beta<\alpha}\ot(\cF_{\le m}(P)\cap (-\infty,a_\beta))\le \sup\limits_{\beta<\alpha}\varphi_{m-1}(\ot(P)+\beta +1)\\=\varphi_{m-1}(\ot(P)+\alpha).
\end{multline*}
From the definition of $b_{i+1}$ it immediately follows that
\begin{equation*}
U_{i+1}\subseteq \cF(G_{\le m-1}(U_i),U_i\cup P)\subseteq\cF_{\le m-1}(U_i\cup P).
\end{equation*}
Applying the induction assumption (\ref{tmeq3}) for $m-1$, we obtain
\begin{equation*}
\ot(U_{i+1})\le\varphi_{m-2}^{k_{m-1}}(\ot(U_i\cup P)+1)\le \varphi^{k_{m-1}}_{m-2}(\ot(P)\oplus\ot(U_i)+1).
\end{equation*}
It follows by induction on $i$ that $\ot(U_i)\le\varphi_{m-2}^{i(k_{m-1}+1)}(\ot(U_0))$. Therefore,
\begin{multline*}
\ot(\cF_{\le m}(P)\cap (\-\infty,a_\alpha))\le \sup\limits_{i\in \mathbb{N}} \ot(U_i)\\ \le \sup\limits_{i\in\mathbb{N}}\varphi_{m-2}^i(\varphi_{m-1}(\ot(P)+\alpha)+1)=\varphi_{m-1}(\ot(P)+\alpha+1),
\end{multline*}
concluding the inductive proof of (\ref{tmeq}).

Now let us use (\ref{tmeq}) to finish the inductive proof of (\ref{tmeq3}). We simply use (\ref{tmeq}) for the case of $a_{\Lambda-1}=+\infty$:
$$\ot(\cF_{\le m}(P))\le \varphi_{m-1}(\ot(P)+\Lambda)<\varphi_{m-1}(\varphi_0^{k+1}(\ot(P)+1))\le \varphi_{m-1}^2(\ot(P)+1).$$
Hence we have (\ref{tmeq3}) with $k_m=2$.

Finally, let us use (\ref{tmeq3}) to give a bound for $\cF(\{g\},\{0\})$. Let $a_0=0$ and $a_{i+1}=g(a_i,\ldots,a_i)$. We prove by induction on $i$ that for $U_i=\cF(\{g\},\{0\})\cap(-\infty,a_i]$ we have $\ot(U_i)< \varphi_{n-1}(0)$. It is obviously the case for $i=0$. For the step of induction we observe that $U_{i+1}\subseteq \cF_{\le n-1}(U_i)$ which allows us to prove the step of induction using (\ref{tmeq3}) for $m=n-1$. Thus $\cF(\{g\},\{0\})\le \varphi_{n-1}(0)$.
\end{proof}

\begin{lemma}\label{linear_are_tame}
For any $\alpha_1,\ldots,\alpha_n\ge 0$ the linear function $g(x_1,\ldots,x_n)=\alpha_1x_1+\ldots+\alpha_nx_n+C$ is tame. 
\end{lemma}
\begin{proof}
Observe that for any $u\subsetneq \{1,\ldots,n\}$ and $\vec{p}\in \mathbb{R}^{n-|u|}$ the set $A(g_{u,\vec{p}})$ consists of a single point. We denote this point as $h_u(\vec{p})$. Notice that $h_u\colon \bbR^{n-|u|}\to\bbR$ is monotone. Consider well-ordered $P\subseteq \mathbb{R}$. Observe that $A(g_u;P)$ is isomorphic to a linearization of a suborder of the wpo $P^{n-|u|}$. Hence $$\begin{aligned} \mathsf{ot}(A(g_u;P))&\le& o(P^{n-|u|})\\ &\le& \underbrace{\mathsf{ot}(P)\otimes \ldots \otimes \mathsf{ot}(P)}\limits_{\mbox{\footnotesize $n-|u|$-times}}\\ &\le& \varphi_0(\underbrace{\mathsf{ot}(P)\oplus\ldots \oplus \mathsf{ot}(P)}\limits_{\mbox{\footnotesize $n-|u|$-times}})\\ &\le &\varphi_0(\mathsf{ot}(P)\omega)\\ &\le &\varphi_0(\varphi_0(\mathsf{ot}(P)+1)).\end{aligned}$$ 
\end{proof}

Theorem \ref{tame_bound} and Lemma \ref{linear_are_tame} together imply Theorem \ref{thm_upper}.

\section{\boldmath Functions $M_n$}\label{sec_M}

In this section we prove Theorem \ref{thm_M}. Recall that for $n\ge 2$ the function $M_n$ is a partial function given by the following recursive algorithm:
\begin{equation*}
    M_n(x) = \begin{cases} -x, &\text{if $x<0$};\\ \underbrace{M_n(x-M_n(x-M_n(\cdots M_n(x-1)\cdots)))}_\text{$n$ instances of $M_n$}/n, &\text{otherwise.}\end{cases}
  \end{equation*}

In other words, given $x\ge 0$ algorithm $M_n$ lets $t_0(x)=1$, then lets $t_i(x)=M_n(x-t_{i-1}(x))$ for $i=1,\ldots, n$, and then it outputs $t_n(x)/n$.

\begin{remark}\label{R_algorithms} Although not that much relevant to our paper, there is a question of what precisely is an algorithm working with arbitrary real numbers. For definiteness we will assume that our model of computation over reals is $\Sigma$-definability in the structure $\mathbb{HF}(\mathbb{R},0,1,+,\times)$, see \cite{MK08,EPS11}. Alternatively, the function $M_n$ can be computed over $\bbQ$ by a standard Turing machine.
\end{remark}

It is not clear a priori that $M_n$ terminates for all inputs $x\in\bbR$. But clearly, whenever $M_n$ terminates on an input $x$, the value $M_n(x)$ that it returns is positive.

For the rest of the section we will assume that $n$ is fixed and we will denote $M_n$ simply as $M$.

\begin{lemma}\label{lem_d}
Suppose $M(x)$ terminates. Then:
\begin{enumerate}
\item For $0\le d<M(x)$, $M(x+d)$ terminates and satisfies $M(x+d) = M(x)-d$.
\item If $x\ge 0$, then $t_i(x)\ge i M(x)$ for each $0\le i\le n$.
\end{enumerate}
\end{lemma}

\begin{proof}
We prove both claims jointly by induction on the depth of the recursive calls. Suppose $M(x)$ terminates, and suppose the claims hold for all recursive calls made by $M(x)$.

We start by proving the second claim. We fix $x\ge 0$ and prove that $t_i(x)\ge iM(x)$ by induction on $0\le i\le n$, in decreasing order. The case $i=n$ is trivial. Hence, suppose $i<n$. By induction assumption $t_{i+1}(x) \ge (i+1)M(x)$. Since $M(x)>0$, if $t_i(x)\ge \frac{i}{i+1} t_{i+1}(x)$ we are done. Therefore, suppose $t_i(x)< \frac{i}{i+1} t_{i+1}(x)$. By induction assumption for the recursive calls, we have $M(x-t_i(x)+d) = t_{i+1}(x) - d$ for all $0\le d<M(x-t_i(x))=t_{i+1}(x)$, so in particular we have
\begin{multline*}
    M(x)=M(x-t_i(x)+t_i(x))=M(x-t_i(x))-t_i(x)\\ =t_{i+1}(x)-t_i(x) \ge (i+1)M(x) - t_i(x),
\end{multline*}
which implies that $t_i(x) \ge i M(x)$, as desired.

We now prove the first claim. If $x<0$, then the claim immediately follows from the definition of the function $M$. Suppose $x\ge 0$. We show by induction on $i$ that $t_i(x+d) = t_i(x) - i d$, for any non-negative $d<M(x)$. The case $i=0$ is trivial, so let $i\ge 1$. By induction on the recursive calls, we have $M(x-t_{i-1}(x) + z) = t_i(x) - z$ for all $z<t_i(x)$. In particular, since $id<iM(x)\le t_i(x)$, we have
\begin{multline*} 
    t_i(x+d) = M(x-t_{i-1}(x+d)+d) = M(x - t_{i-1}(x) + id)\\ = M(x-t_{i-1}(x)) - id = t_i(x) - id.
\end{multline*}

Therefore, for any $0\le d <M(x)$ we have $M(x+d)= t_n(x+d)/n = (t_n(x) - nd)/n = M(x) - d$, as desired.
\end{proof}

\begin{proposition}\label{M_is_total}
Algorithm $M$ terminates on all real inputs.
\end{proposition}

\begin{proof}
Let $S\subset\bbR$ be the set of all inputs on which $M$ does not terminate, and suppose for a contradiction that $S\neq\emptyset$. Let $x= \inf S$. Since $M(x-\eps)$ terminates for all $\eps>0$, and since $M$ outputs a positive value whenever it terminates, it follows that $M(x-t_i(x))$ terminates for all $i\le n$. Hence, $M(x)$ itself terminates. But then Lemma \ref{lem_d} implies that $M(x+d)$ terminates for all $0\le d<M(x)$, yielding a contradiction to the definition of $x$.
\end{proof}

\begin{lemma}\label{lem_leq}
For every $x\le y$ we have $x+M(x)\leq y+M(y)$, with equality if and only if $y<x+M(x)$.
\end{lemma}

\begin{proof}
Let $d=y-x$. If $d<M(x)$, then Lemma \ref{lem_d} implies $y+M(y) = y+M(x+d) = y+M(x) - d = x+M(x)$. Otherwise, we have $y+M(y) > y = x+d \ge x+M(x)$.
\end{proof}

Define the set
\begin{equation*}
    \tameF = \{x + M(x) : x\in\mathbb{R}\}.
\end{equation*}
Then Lemma \ref{lem_leq} implies:

\begin{corollary}\label{cor_M_zero}
$\cF'_n$ is the set of points at which $M$ tends to zero from the left:
\begin{equation*}
\cF'_n=\{x\in\bbR:\lim_{y\to x^-} M(y)=0\}.
\end{equation*}
\end{corollary}

\begin{lemma}\label{lem_tii}
 $$\frac{t_i(x)}{i}\ge \frac{t_{i+1}(x)}{i+1}\text{, for each $x\ge 0$ and $1\le i\le n-1$.}$$
\end{lemma}

\begin{proof}
By Lemmas \ref{lem_d} and \ref{lem_leq} we have
\begin{equation*}
    \frac{t_i(x)}{i}\ge M(x) \ge M(x-t_i(x)) - t_i(x) = t_{i+1}(x) - t_i(x),
\end{equation*}
which implies the claim.
\end{proof}



\begin{lemma}\label{all_tame_are_fusibles}
Every element of $\tameF$ is an $n$-fusible number, meaning $\tameF\subseteq\cF_n$.
\end{lemma}

\begin{proof}
By induction on the depth of the recursive calls in the computation of $M(x)$ we show that $x+M(x)$ is an $n$-fusible number. In the base case, when $x<0$ we have $x+M(x)=0$ which is an $n$-fusible number. Consider the case $x\ge 0$. For each $1\le i\le n$ we put $z_i= x-t_{i-1}(x)+t_i(x)$. Since $z_i=x-t_{i-1}(x)+M(x-t_{i-1}(x))$, by the induction hypothesis they are $n$-fusible numbers. By Lemma \ref{lem_leq} we have $z_i\le x+M(x)$. Finally, observe that $x+M(x)=f_n(z_1,\ldots,z_n)$, so $x+M(x)$ is $n$-fusible.
\end{proof}

All that remains is to prove that $\ot(\tameF) \ge \varphi_{n-1}(0)$.

\begin{lemma}\label{lem_tlinear}
For every $1\le i \le n$, every $x$, and every $0\le d<t_i(x)/i$ we have $t_i(x+d) = t_i(x) - di$.
\end{lemma}

\begin{proof}
By induction on $i$. First let $i=1$. Since $0\le d< t_1(x) = M(x-1)$, Lemma \ref{lem_leq} implies that $t_1(x+d) = M(x+d-1) = M(x-1)-d = t_1(x) -d$.

Now let $2\le i \le n$. By Lemma \ref{lem_tii} we have $0\le d < t_i(x)/i \le t_{i-1}(x)/(i-1)$, so by applying induction on $i$ and then Lemma \ref{lem_leq}, we have $t_i(x+d) = M(x+d-t_{i-1}(x+d)) = M(x+d-t_{i-1}(x)+d(i-1))=M(x-t_{i-1}(x) + di) = M(x-t_{i-1}(x))-di=t_i(x)-di$.
\end{proof}

\begin{lemma}\label{lem_invtlinear}
    For every $1\le i \le n$, every $x$, and every $0\le d\le x$ we have $t_i(x-d) \le  t_i(x) + di$.
\end{lemma}
\begin{proof}
Indeed, if $d<t_i(x-d)/i$, then we have $t_i(x-d) =  t_i(x) + di$ by Lemma 
\ref{lem_tlinear}. Otherwise $t_i(x-d) \le  di$ and hence $t_i(x-d) <  t_i(x) + di$
\end{proof}

For $x\ge 0$ and $0<i<n$ let $U_i(x)=t_i(x)/i$. We also put $U_0(x)=+\infty$.

\begin{lemma}\label{U_and_M}
    For every $0\le i<n$ and $x\ge 0$ we have $U_i(x)>M(x)$.
\end{lemma}
\begin{proof}
In view of Lemma \ref{lem_tii} it is enough to prove that $U_{n-1}(x)>M(x)$, for $x\ge 0$. Suppose for a contradiction that there exists $x$ such that $U_{n-1}(x) = M(x)$ (by Lemma \ref{lem_d} we cannot have $U_{n-1}(x)<M(x)$).  Take a recursion-minimal such $x$. Let $y=x-t_{n-1}(x) = x - (n-1)M(x)$, and note that $M(x)$ is recursively computed as $M(x)=M(y)/n$. By the recursion-minimality of $x$, we have either $y\ge 0$ and $t_{n-1}(y) > (n-1)M(y)$, or $y<0$. We cannot have $y\ge 0$ and $t_{n-1}(y)>(n-1)M(y)$, since by Lemma \ref{lem_invtlinear} we have $t_{n-1}(y) = t_{n-1}(x-(n-1)M(x)) \le t_{n-1}(x) + (n-1)^2M(x) = n(n-1)M(x) = (n-1)M(y)$. On the other hand, if $y<0$, then $nM(x)=M(y)=-y=(n-1)M(x)-x$, so $M(x)=-x\le 0$, contradicting the fact that values of $M$ are always positive.
\end{proof}

For $x\ge 0$ and $0\le i<n$ define the interval
\begin{equation*}
J_i(x) = [x-t_i(x),x+U_i(x)),
\end{equation*}
and define the linear transformation
\begin{equation*}
l_i^x(y)= \frac{ix}{i+1}+\frac{t_i(x)}{i+1}+\frac{y}{i+1}.
\end{equation*}
Note that each $l_i^x$ maps $J_i(x)$ onto $[x,x+U_i(x))$.

\begin{lemma}\label{lem_l} For any $x\ge 0$, $0\le i<n$, and $y\in J_i(x)$ we have 
$$t_{i+1}(l^x_i(y))=M(y).$$
\end{lemma}
\begin{proof}
If $i=0$ then
$$t_1(l^x_0(y))=t_1(y+1)=M(y+1-t_0(y+1))=M(y+1-1)=M(y).$$

Now assume that $i>0$. Let $x_0=x-t_i(x)$ and $d=y-x_0$. Thus
$$l_i^x(y)=\frac{i(x_0+t_i(x))}{i+1}+\frac{t_i(x)}{i+1}+\frac{x_0+d}{i+1}= x_0+t_i(x)+\frac{d}{i+1}.$$
Notice that $$d/(i+1)=\frac{y-x_0}{i+1}=\frac{y-x+t_i(x)}{i+1}<\frac{x+t_i(x)/i-x+t_i(x)}{i+1}=t_i(x)/i.$$ 
Hence by Lemma \ref{lem_tlinear} $$t_i(l_i^x(y))=t_i\Bigl(x_0+t_i(x)+\frac{d}{i+1}\Bigr)=t_i\Bigl(x+\frac{d}{i+1}\Bigr)=t_i(x)-\frac{id}{i+1}.$$
Therefore
\begin{multline*}
t_{i+1}(l^x_i(y))=M(l^x_i(y)-t_i(l^x_i(y)))=M\Bigl(x_0+t_i(x)+\frac{d}{i+1}-t_i(x)+\frac{id}{i+1}\Bigr)\\=M(x_0+d)=M(y).
\end{multline*}
\end{proof}

For $x\in\mathbb{R}$, denote $o_{\mathit{lim}}(x)=\inf_{\varepsilon>0} \ot(\cF_n'\cap(x-\eps,x))$. Also denote $o_{\mathit{lim}}(+\infty)=\inf_{x>0} \ot(\cF_n'\cap[x,+\infty))$. Hence, if $x$ is not a limit point of $\cF_n'$ then $o_{\mathit{lim}}(x)=0$, while if $x$ is a limit point of $\cF_n'$ and $\ot(\cF_n'\cap(-\infty,x)) = \omega^{\alpha_1} + \cdots + \omega^{\alpha_k}$ in Cantor Normal Form, then $o_{\mathit{lim}}(x)=\omega^{\alpha_k}$.

For an ordinal $\alpha$ and $i\ge 0$ let $\eta_i(\alpha)$ be the least limit ordinal of the form $\varphi_i(\beta)$, such that $\varphi_i(\beta)>\alpha$.

\begin{lemma} \label{ot_calc} For any $0\le i<n$ and $x\ge 0$ we have 
$$o_{\mathit{lim}}(x+U_i(x))\ge \eta_{n-i-1}(o_{\mathit{lim}}(x)).$$\end{lemma}
\begin{proof}
We prove by induction on $0\le i<n$ in decreasing order that the lemma holds for all $x\ge 0$.


Since $\cF_n'$ is well-ordered its closure $\overline{\cF}_n'$ is also well-ordered. Let $\Lambda = \ot(\overline{\cF}_n'\cap J_i(x))$, and let $\langle a_\alpha\mid \alpha <\Lambda\rangle$ be an enumeration of elements of $\overline{\cF}_n'\cap J_i(x)$ in increasing order. Observe that by Lemma \ref{lem_tlinear}, for any $x\le y<x+U_i(x)$ we have $y+U_i(y)=x+U_i(x)$. Since by Lemma \ref{U_and_M} $y+M(y)<y+U_i(y)$ we see that there are elements of $\cF'_n$ in any left neighbourhood of $x+U_i(x)$. In other words, $x+U_i(x)$ is a limit point of $\cF'_n$. In particular this implies that both $\Lambda$ and $o_{\mathit{lim}}(x+U_i(x))$ are limit ordinals. 

In the case of $i=n-1$, by Lemma \ref{lem_l} we have 
\begin{equation}\label{M_iso}M(l^x_{n-1}(y))=\frac{M(y)}{n}\text{, for $y\in J_{n-1}(x)$}.\end{equation}
Therefore, Corollary \ref{cor_M_zero} implies that, for every $y\in J_{n-1}(x)$, we have $y\in \cF'_n$ iff $l^x_{n-1}(y)\in \cF'_n$. Consider the points $x_0=x$, $x_{j+1}=l_{n-1}^x(x_j)$. We have $o_{\mathit{lim}}(x_{j+1})=o_{\mathit{lim}}(x_{j})$. Since $\lim_{j\to\infty} x_j= x+U_{n-1}(x)$, we conclude that $o_{\mathit{lim}}(x+U_{n-1}(x))\ge o_{\mathit{lim}}(x)\omega$. In general, for every ordinal $\alpha>0$, the ordinal $\alpha\omega$ is a power of $\omega$, so $\alpha\omega \ge \eta_0(\alpha)$. Hence, in our case
$$o_{\mathit{lim}}(x+U_{n-1}(x))\ge o_{\mathit{lim}}(x)\omega\ge\eta_0(o_{\mathit{lim}}(x)).$$

Now consider the case of $i<n-1$. By Lemma \ref{lem_l} we have 
\begin{equation}\label{U_M_iso}U_{i+1}(l^x_{i}(y))=\frac{M(y)}{i+1}\text{, for $y\in J_i(x)$}.\end{equation}

We prove by induction on $\alpha$ that \begin{equation}\label{o_lim_alpha} o_{\mathit{lim}}(l_i^x(a_{1+\alpha}))\ge  \varphi_{n-i-2}(\alpha).\end{equation} Clearly $$a_{1+\alpha}=\sup \{a_{\beta}+M(a_\beta) \mid \beta<1+\alpha\}.$$ Hence by (\ref{U_M_iso}) we have $$\begin{aligned} l_i^x(a_{1+\alpha})&=&\sup \{l_i^x(a_{\beta}+M(a_\beta)) \mid \beta<1+\alpha\}\\ &=& \sup \{l_i^x(a_{\beta})+\frac{M(a_\beta)}{i+1} \mid \beta<1+\alpha\} \\ &=& \sup \{l_i^x(a_{\beta})+U_{i+1}(l_i^x(a_\beta)) \mid \beta<1+\alpha\}.\end{aligned}$$ Therefore
\begin{equation}\label{eq_olim_temp}
o_{\mathit{lim}}(l_i^x(a_{1+\alpha}))\ge \sup \{o_{\mathit{lim}}(l_i^x(a_{\beta})+U_{i+1}(l_i^x(a_\beta))) \mid \beta<1+\alpha\}.
\end{equation}
By the induction assumption on $i$ we have $$o_{\mathit{lim}}(l_i^x(a_{\beta})+U_{i+1}(l_i^x(a_\beta))) \ge \eta_{n-i-2}(o_{\mathit{lim}}(l_i^x(a_\beta))).$$ Furthermore, by induction on $\alpha$ we have $o_{\mathit{lim}}(l_i^x(a_{1+\beta}))\ge  \varphi_{n-i-2}(\beta)$ for every $\beta<\alpha$. Substituting into (\ref{eq_olim_temp}), we obtain
 $$o_{\mathit{lim}}(l_i^x(a_{1+\alpha}))\ge\sup (\{\eta_{n-i-2}(0)\}\cup \{\eta_{n-i-2}(\varphi_{n-i-2}(\beta))\mid \beta<\alpha\})=\varphi_{n-i-2}(\alpha),$$ completing the proof of (\ref{o_lim_alpha}).

Now we use (\ref{o_lim_alpha}) to show that $\Lambda$ is a fixed point of $\varphi_{n-i-2}$. Since $\varphi_{n-i-2}$ is a normal function and $\Lambda$ is a limit ordinal, it is enough to show that for any $\alpha<\Lambda$ we have $\varphi_{n-i-2}(\alpha)\le\Lambda$. Consider $\alpha<\Lambda$. Since $\Lambda\ge \omega$,  we have $1+\alpha<\Lambda$. And since $x<l_i^x(a_{1+\alpha})<x+U_i(x)$, we conclude that $\Lambda>o_{\mathit{lim}}(l_i^x(a_{1+\alpha}))$. Thus by (\ref{o_lim_alpha}) we have $\Lambda\ge\varphi_{n-i-2}(\alpha)$. 

Given that $\Lambda>o_{\mathit{lim}}(x)$ and is a fixed point of $\varphi_{n-i-2}$ we conclude that $\Lambda\ge \eta_{n-i-1}(o_{\mathit{lim}}(x))$. Since $x+U_i(x)$ is a limit point of $\cF'_n$, we have $o_{\mathit{lim}}(x+U_i(x))=\Lambda$. Thus $o_{\mathit{lim}}(x+U_i(x))\ge\eta_{n-i-1}(o_{\mathit{lim}}(x))$.\end{proof}

Applying Lemma \ref{ot_calc} in the case of $i=0$ we see that $$\ot(\cF'_n)\ge o_{\mathit{lim}}(U_0(0))\ge \eta_{n-1}(0)=\varphi_{n-1}(0).$$
Thus taking into account Theorem \ref{tame_bound} and Lemma \ref{all_tame_are_fusibles} we have
$$\varphi_{n-1}(0)\ge \ot(\cF_n)\ge \ot(\cF'_n)\ge \varphi_{n-1}(0),$$
as desired. 
 
 \begin{remark}Theorem \ref{thm_M} implies that the statement ``for all $n$ the set of rationals $\mathcal{F}_n$ is well-ordered'' is not provable in systems of second-order arithmetic whose proof-theoretic ordinal is  $\le\varphi_\omega(0)$. Natural examples of systems with proof-theoretic strength $\varphi_\omega(0)$ are $\Delta^1_1\textsf{-CR}$ \cite{feferman64} (see also \cite{sep-proof-theory}) and  $\Pi^1_1\textsf{-BI}_0$ (it is fairly easy to show that $|\Pi^1_1\textsf{-BI}_0|=\varphi_\omega(0)$ using \cite[Main~Theorem]{jaeger_strahm99} and \cite[Lemma 2.10 and Theorem 5.11]{PW22}). It is natural to conjecture that for any fixed natural $n$ these two systems are capable of showing that the set $\mathcal{F}_n$ is well-ordered. Unfortunately, our current proof of Theorem \ref{thm_upper} relies on the fact that the set under consideration is already well-ordered (a fact that we prove using Kruskal's tree theorem that is outside of the reach of the systems of this proof-theoretic strength).  However, it might be possible to make a proof that is formalizable in $\Delta^1_1\textsf{-CR}$ and $\Pi^1_1\textsf{-BI}_0$. For example, one potential route would be to prove the well-foundedness of $\mathcal{F}_n$ by giving a recursive embedding of the set into the standard ordinal notation system for $\varphi_{n-1}(0)$.
\end{remark}

\section{Functions generating sets of high order type}
\label{sec_cont}
In this section we prove Theorem \ref{thm_cont}. The first step is to build a suitable function on ordinals:

\begin{lemma}\label{lem_phi-bar-star}
Let $n\ge 3$, and let $\Lambda=\varphi(1,\underbrace{0,\ldots,0}\limits_{\mbox{\footnotesize $n$ times}})$. Then there exists a bijection $\bar{\varphi}^\star:\Lambda^n\to\Lambda\setminus\{0\}$ that is monotone (meaning, $\bar{\varphi}^\star(\alpha_1,\ldots,\alpha_n)\le\bar{\varphi}^\star(\beta_1,\ldots,\beta_n)$ whenever $\alpha_i\le \beta_i$ for all $1\le i\le n$), and such that $\bar{\varphi}^\star(\alpha_1,\ldots,\alpha_n)>\max\{\alpha_1,\ldots,\alpha_n\}$.
\end{lemma}

Transfinite induction then implies that every ordinal in $\Lambda$ can be constructed by starting from $0$ and repeatedly applying $\bar{\varphi}^\star$ on previously constructed ordinals. 

The idea of our construction of $\bar\varphi^\star$ is to start from a fixed-point free variant $\bar\varphi$ of the $n$-ary Veblen function, and then contracting its range to remove gaps. To analyze the resulting function $\bar\varphi^\star$, we make use of a simple comparison criterion (Proposition \ref{Veblen_star_comp} below) that exactly mirrors the criterion for $\bar \varphi$ (Lemma \ref{Veblen_comp} below). This comparison criterion makes our construction very similar to the tree-based ordinal notation by Jervell \cite{jervell_CiE,jervell_phil}, though Jervell only provides a rough proof sketch of the connection between his tree-based notation and the standard Veblen function $\varphi$. We also note that a somewhat different function generating $\Lambda$ was constructed by Schmidt \cite[Theorem~4.8]{Sch20}.

\begin{proof}[Proof of Lemma~\ref{lem_phi-bar-star}]

First for $n\ge 2$ we define a fixed-point-free variant $\bar\varphi$ of the $n$-ary Veblen function. Let $$C(\alpha_1,\ldots,\alpha_{n-1})=\{\varphi(\alpha_1,\ldots,\alpha_{n-1},\beta)\mid \beta\in\mathsf{On}\}.$$
Let
\begin{multline*}
C^{-}(\alpha_1,\ldots,\alpha_{n-2},\alpha_{n-1})=\\C(\alpha_1,\ldots,\alpha_{n-2},\alpha_{n-1})\setminus (C(\alpha_1,\ldots,\alpha_{n-2},\alpha_{n-1}+1)\cup \{\alpha_1,\ldots,\alpha_{n-1}\}).
\end{multline*}
The fixed-point free $n$-ary Veblen function $\bar{\varphi}\colon  \mathsf{On}^n\to\mathsf{On}$ is defined as follows
$$\bar{\varphi}(\alpha_1,\ldots,\alpha_n)=\text{``$\alpha_n$-th element of $C^-(\alpha_1,\ldots,\alpha_{n-2},\alpha_{n-1})$''}.$$

\begin{lemma}\label{Veblen_comp}
For any ordinals $\alpha_1,\ldots,\alpha_n,\beta_1,\ldots,\beta_n$ we have $\bar{\varphi}(\alpha_1,\ldots,\alpha_n)<\bar{\varphi}(\beta_1,\ldots,\beta_n)$ if and only if at least one of the following conditions holds:
\begin{enumerate}
    \item for some $1\le i\le n$ we have $\bar{\varphi}(\alpha_1,\ldots,\alpha_n)\le \beta_i$;
    \item  for some $1\le i\le n$ we have $\alpha_1=\beta_1$, $\ldots$, $\alpha_{i-1}=\beta_{i-1}$, $\alpha_{i}<\beta_i$, and $\alpha_{i+1},\ldots,\alpha_n<\bar{\varphi}(\beta_1,\ldots,\beta_n)$.
\end{enumerate}
\end{lemma}
\begin{proof}
  Observe that $C(\alpha_1,\ldots,\alpha_{n-2},\alpha_{n-1}+1)$ is precisely the set of all fixed points of the function $x\mapsto \varphi(\alpha_1,\ldots,\alpha_{n-1},x)$. The value $\bar \varphi (\alpha_1,\ldots,\alpha_{n-1},\alpha_n)$ is $\alpha_n$-th element of $C^{-}(\alpha_1,\ldots,\alpha_{n-2},\alpha_{n-1})$, i.e. it is $\alpha_n'$-th element of $C(\alpha_1,\ldots,\alpha_{n-2},\alpha_{n-1})$, for certain $\alpha_n'\ge\alpha_n$, $\alpha_n'\not\in C(\alpha_1,\ldots,\alpha_{n-2},\alpha_{n-1}+1)$. Therefore \begin{equation}\label{Veblen_comp_1}\bar \varphi (\alpha_1,\ldots,\alpha_{n-1},\alpha_n)=\varphi (\alpha_1,\ldots,\alpha_{n-1},\alpha_n')>\alpha_n'\ge \alpha_n.\end{equation}

  It is fairly easy to see that \begin{equation}\label{Veblen_comp_2}\varphi(\alpha_1,\ldots,\alpha_{n-1},\alpha_n)\ge \{\alpha_1,\ldots,\alpha_{n}\}\end{equation} and hence  $$\min C(\alpha_1,\ldots,\alpha_{n-2},\alpha_{n-1})\ge \max \{\alpha_1,\ldots,\alpha_{n-1}\}.$$
  Therefore
  $$\min C^{-}(\alpha_1,\ldots,\alpha_{n-2},\alpha_{n-1})> \max \{\alpha_1,\ldots,\alpha_{n-1}\}$$
  and hence 
  \begin{equation}\label{Veblen_comp_3}\bar \varphi (\alpha_1,\ldots,\alpha_{n-1},\alpha_n)>\alpha_i\text{, for $1\le i<n$}.\end{equation}

  Combining (\ref{Veblen_comp_1}) and (\ref{Veblen_comp_3}) we see that
    \begin{equation}\label{Veblen_comp_4}\bar \varphi (\alpha_1,\ldots,\alpha_{n-1},\alpha_n)>\alpha_i\text{, for $1\le i\le n$}.\end{equation}
   Therefore if $\bar{\varphi}(\alpha_1,\ldots,\alpha_n)\le \beta_i$, then $\bar{\varphi}(\alpha_1,\ldots,\alpha_n)< \bar{\varphi}(\beta_1,\ldots,\beta_n)$.

  Suppose we have $\alpha_1=\beta_1,\ldots,\alpha_{i-1}=\beta_{i-1}$, $\alpha_i<\beta_i$, and $\alpha_{i+1},\ldots,\alpha_n<\bar\varphi(\beta_1,\ldots,\beta_n)$. We claim that $\bar\varphi(\alpha_1,\ldots,\alpha_n)<\bar\varphi(\beta_1,\ldots,\beta_n)$. By definition $\bar\varphi(\alpha_1,\ldots,\allowbreak \alpha_n)=\varphi(\alpha_1,\ldots,\alpha_{n-1},\alpha_n')$ and $\bar\varphi(\beta_1,\ldots,\beta_n)=\varphi(\beta_1,\ldots,\beta_{n-1},\beta_n')$, for some $\beta_n'$. Proposition~\ref{prop_remove_limits} implies that $\alpha_n'<\alpha_n+\omega$ and $\beta_n'<\beta_n+\omega$. Hence $\alpha_n'<\varphi(\beta_1,\ldots,\beta_{n-1},\beta_n')$. Therefore by the the definition of $\varphi$ we will have $\varphi(\alpha_1,\ldots,\allowbreak\alpha_{n-1},\alpha_n')<\varphi(\beta_1,\ldots,\beta_{n-1},\beta_n')$, which concludes the proof of the claim.

  We have proven that if either condition (1) or condition (2) holds, then $\bar{\varphi}(\alpha_1,\ldots,\allowbreak\alpha_n)<\bar{\varphi}(\beta_1,\ldots,\beta_n)$. Now suppose that neither (1) nor (2) holds. If $(\alpha_1,\ldots,\alpha_n)=(\beta_1,\ldots,\beta_n)$ then trivially $\bar{\varphi}(\alpha_1,\ldots,\alpha_n)=\bar{\varphi}(\beta_1,\ldots,\beta_n)$, and the claim follows. Hence, assume $\alpha_i\neq\beta_i$ for some $i$. Let $i$ be the minimal such index. If $\alpha_j\ge\bar\varphi(\beta_1,\ldots,\beta_n)$ for some $j$, then (\ref{Veblen_comp_4}) implies $\bar{\varphi}(\alpha_1,\ldots,\alpha_n)>\bar{\varphi}(\beta_1,\ldots,\beta_n)$, so the claim follows as well. Hence, suppose $\alpha_j < \bar\varphi(\beta_1,\ldots,\beta_n)$ for every $j$. Since (2) does not hold, we must have $\beta_i<\alpha_i$. Furthermore, since (1) does not hold, we must have $\beta_j<\bar\varphi(\alpha_1,\ldots,\alpha_n)$ for every $j$, in particular for $j>i$. Hence, (2) holds with $\alpha_1,\ldots,\alpha_n$ and $\beta_1,\ldots,\beta_n$ switched. As we have shown, this implies $\bar{\varphi}(\alpha_1,\ldots,\alpha_n)>\bar{\varphi}(\beta_1,\ldots,\beta_n)$.
\end{proof}

Continuing the proof of Lemma \ref{lem_phi-bar-star}, let $A$ be the well-ordering consisting of all closed terms built from the constant $0$ and the $n$-ary function $\bar{\varphi}$, where elements of $A$ are compared according to their ordinal values. Let $\Lambda$ be the order type of $A$. Let $\bar{\varphi}^\star\colon \Lambda^n\to\Lambda\setminus \{0\}$ be the following function. Given ordinals $\alpha_1,\ldots,\alpha_n<\Lambda$ such that terms $t_1,\ldots,t_n\in A$ lie in the positions $\alpha_1,\ldots,\alpha_n$ respectively, the value $\bar\varphi^\star(\alpha_1,\ldots,\alpha_n)$ is the position of the term $\bar\varphi(t_1,\ldots,t_n)$. Notice that immediately from  Lemma \ref{Veblen_comp} it follows that $\bar{\varphi}^{\star}$ is a strictly monotone bijection between $\Lambda^n$ and $\Lambda\setminus \{0\}$, for which the analogue of Lemma \ref{Veblen_comp} holds:
\begin{proposition}\label{Veblen_star_comp}
For any ordinals $\alpha_1,\ldots,\alpha_n,\beta_1,\ldots,\beta_n$ we have $\bar{\varphi}^\star(\alpha_1,\ldots,\alpha_n)<\bar{\varphi}^\star(\beta_1,\ldots,\beta_n)$ iff at least one of the following conditions holds:
\begin{enumerate}
    \item for some $1\le i\le n$ we have $\bar{\varphi}^\star(\alpha_1,\ldots,\alpha_n)\le \beta_i$;
    \item  for some $1\le i\le n$ we have $\alpha_1=\beta_1$, $\ldots$, $\alpha_{i-1}=\beta_{i-1}$, $\alpha_{i}<\beta_i$, and $\alpha_{i+1},\ldots,\alpha_n<\bar{\varphi}^{\star}(\beta_1,\ldots,\beta_n)$.
\end{enumerate}
\end{proposition}

Since $\bar{\varphi}^{\star}$ is a bijection between $\Lambda^n$ and $\Lambda\setminus \{0\}$, Proposition \ref{Veblen_star_comp} implies the following (compare with Proposition \ref{prop_veblen_equiv}):
\begin{lemma}\label{Veblen_star_form1}
    For every $\alpha_1,\ldots,\alpha_n$, the value $\bar{\varphi}^\star(\alpha_1,\ldots,\alpha_n)$ is the least ordinal $\gamma$ strictly above $\alpha_1,\ldots,\alpha_n$ such that for any $1\le i\le n$, $\beta_i<\alpha_i$ and $\beta_{i+1},\ldots,\beta_{n}<\gamma$:
    $$\gamma>\bar\varphi^\star(\alpha_1,\ldots,\alpha_{i-1},\beta_i,\beta_{i+1},\ldots,\beta_n).$$
\end{lemma}
The definition of $\bar{\varphi}^\star$ immediately implies $\bar{\varphi}^\star(\alpha_1,\ldots,\alpha_n)\le \bar{\varphi}(\alpha_1,\ldots,\alpha_n)$. On the other hand, we have the following lower bounds for $\bar{\varphi}^\star$ in terms of $\varphi$:

\begin{lemma}\label{lem_phi-bar-star_lower}
For every $\alpha_1,\ldots,\alpha_n<\Lambda$ we have:
\begin{enumerate}
    \item $\bar{\varphi}^\star(0,\ldots,0,\alpha_n) = \alpha_n+1$.
    
    \item $\bar{\varphi}^\star(0,\ldots,0,\alpha_{n-1},\alpha_n) \ge \alpha_n+\omega^{\alpha_{n-1}}$.
    
    \item If $\alpha_i>0$ for some $i<n-1$ then $\bar{\varphi}^\star(\alpha_1,\ldots,\alpha_n)$ is an $\varepsilon$-number.
    
    \item $\bar{\varphi}^\star(0,\ldots,0,1+\alpha_{n-2},\alpha_{n-1},\alpha_n)\ge \varphi(0,\ldots,0,\alpha_{n-2},\alpha_{n-1},\alpha_n)$.
    
    \item If $\alpha_i>0$ for some $i<n-2$ then $\bar{\varphi}^\star(\alpha_1,\ldots,\alpha_n)\ge \varphi(\alpha_1,\ldots,\alpha_n)$.
\end{enumerate}
\end{lemma}

\begin{proof}
We proceed by transfinite induction on the lexicographic order of $\Lambda^n=\Lambda\times\cdots\times\Lambda$.

For item (1), suppose $\gamma=\bar{\varphi}^\star(\beta_1,\ldots,\beta_n) < \bar{\varphi}^\star(0,\ldots,0,\alpha_n)$. We will show that $\gamma\le \alpha_n$. Proposition \ref{Veblen_star_comp} implies that either $\gamma\le \alpha_n$, or else $\beta_1=\cdots=\beta_{n-1}=0$ and $\beta_n<\alpha_n$. In the latter case, transfinite induction implies $\gamma=\beta_n+1 \le \alpha_n$ again.

For item (2), let $\beta=\bar{\varphi}^\star(0,\ldots,0,\alpha_{n-1},\alpha_n)$, and note that $\beta>\alpha_n$. If $\beta\ge\alpha_n+\omega^{\alpha_{n-1}}$ we are done. Otherwise, there exist $\gamma<\alpha_{n-1}$ and $m\in\bbN$ such that $\alpha_n + \omega^\gamma m < \beta\le \alpha_n + \omega^\gamma (m+1)$. Transfinite induction and Proposition \ref{Veblen_star_comp} imply
\begin{equation*}
    \beta\le \alpha_n+\omega^\gamma(m+1)\le\bar{\varphi}^\star(0,\ldots,0,\gamma,\alpha_n+\omega^\gamma m)<\bar{\varphi}^\star(0,\ldots,0,\alpha_{n-1},\alpha_n)=\beta,
\end{equation*}
contradiction.

For item (3), let $\beta=\bar{\varphi}^\star(\alpha_1,\ldots,\alpha_n)$ where $\alpha_i>0$ for some $i<n-1$. It is enough to show that $\beta>\omega^\gamma$ for every $\gamma<\beta$, since this implies $\beta\ge \omega^\beta$. Let $\gamma<\beta$. Then Proposition \ref{Veblen_star_comp} and item (2) imply $\beta >\bar{\varphi}^\star(0,\ldots,0,\gamma,0)\ge \omega^\gamma$, as desired.

For item (4), let $\beta=\bar{\varphi}^\star(0,\ldots,0,1+\alpha_{n-2},\alpha_{n-1},\alpha_n)$. Transfinite induction and Proposition \ref{Veblen_star_comp} imply the following:
\begin{itemize}
\item For every $\alpha'_n<\alpha_n$ we have $\varphi(0,\ldots,0,\alpha_{n-2},\alpha_{n-1},\alpha'_n) \le \bar{\varphi}^\star(0,\ldots,0,1+\alpha_{n-2},\alpha_{n-1},\alpha'_n)<\beta$.
\item For every $\alpha'_{n-1}<\alpha_{n-1}$ and $\alpha'_n<\beta$ we have $\varphi(0,\ldots,0,\alpha_{n-2},\alpha'_{n-1},\alpha'_n) \le \bar{\varphi}^\star(0,\ldots,0,1+\alpha_{n-2},\alpha'_{n-1},\alpha'_n)<\beta$.
\item For every $\alpha'_{n-2}<\alpha_{n-2}$ and $\alpha'_{n-1},\alpha'_n<\beta$ we have $\varphi(0,\ldots,0,\alpha'_{n-2},\alpha'_{n-1},\allowbreak\alpha'_n) \le \bar{\varphi}^\star(0,\ldots,0,1+\alpha'_{n-2},\alpha'_{n-1},\alpha'_n)<\beta$.
\end{itemize}
Furthermore, by item (3), $\beta$ is of the form $\omega^\gamma$. Hence, Proposition \ref{prop_veblen_equiv} implies that $\varphi(0,\ldots,0,\alpha_{n-2},\alpha_{n-1},\alpha_n)\le \beta$, as desired.

Finally, for item (5), let $\beta=\bar{\varphi}^\star(\alpha_1,\ldots,\alpha_n)$, where $\alpha_i>0$ for some $i<n-2$. Let $1\le j\le n$, and suppose $\alpha'_j<\alpha_j$ and $\alpha'_{j+1},\ldots,\alpha'_n<\beta$. We claim that $\gamma=\varphi(\alpha_1,\ldots,\alpha_{j-1},\alpha'_j,\ldots,\alpha'_n)<\beta$. If the first $n-3$ elements among $\alpha_1,\alpha_{j-1},\alpha'_j,\ldots,\alpha'_n$ are not all $0$, then transfinite induction on item (5), together with Proposition \ref{Veblen_star_comp}, imply that $\gamma\le \bar{\varphi}^\star(\alpha_1,\ldots,\alpha_{j-1},\alpha'_j,\ldots,\alpha'_n)<\beta$. Otherwise, by item (4) and Proposition \ref{Veblen_star_comp} we have
\begin{equation*}
\gamma=\varphi(0,\ldots,0,\alpha'_{n-2},\alpha'_{n-1},\alpha'_n)\le\bar{\varphi}^\star(0,\ldots,0,1+\alpha'_{n-2},\alpha'_{n-1},\alpha'_n)<\beta
\end{equation*}
(because $1+\alpha'_{n-1}<\beta$, since $\beta$ is an $\varepsilon$-number). Hence, Proposition \ref{prop_veblen_equiv} implies that $\varphi(\alpha_1,\ldots,\alpha_n)\le \beta$, as desired.
\end{proof}

Lemma \ref{lem_phi-bar-star_lower} implies that $\Lambda=\varphi(1,\underbrace{0,\ldots,0}\limits_{\mbox{\footnotesize $n$ times}})$, as desired.
\end{proof}

The next step in proving Theorem \ref{thm_cont} is to embed $\Lambda$ into the reals.

\begin{lemma}\label{Lambda_emb} There exists an order preserving embedding $e$ of $\Lambda$ into $\mathbb{R}$ such that $\sup\limits_{\alpha<\lambda}e(\alpha)<e(\lambda)$, for any limit ordinal $\lambda<\Lambda$.\end{lemma}
\begin{proof}
It is easily shown by ordinal induction that any countable ordinal has an order-preserving embedding into $\bbR$, and Proposition \ref{prop_remove_limits} allows us to remove any limit elements.

For the sake of completeness, we also offer an explicit embedding (which also works for any countable ordinal). Fix an enumeration $\alpha_0,\alpha_1,\ldots$ (indexed by natural numbers) of elements of $\Lambda$ such that $\alpha_0=0$. We define values $e(\alpha_i)$ by induction on $i$. We put $e(\alpha_0)=0$. To define $e(\alpha_{i+1})$ we find $0\le k\le i$ such that $\alpha_{k}=\max\{\alpha_j\mid j\le i\text{ and }\alpha_j<\alpha_{i+1}\}$ and next we put $e(\alpha_{i+1})=e(\alpha_k)+3^{-k}+3^{-i}$. 

By induction on $n$ we simultaneously prove the following two claims: 
\begin{enumerate}
\item $e$ is order preserving on the set $\{\alpha_0,\ldots,\alpha_n\}$;
\item for any $i,j\le n$, if $\alpha_i<\alpha_j$ and $\alpha_i,\alpha_j$ are neighbours in $\{\alpha_0,\ldots,\alpha_{n}\}$, then $e(\alpha_j)-e(\alpha_i)\ge 3^{-i}+2\cdot 3^{-n}+3^{-j}$. 
\end{enumerate}
The base holds for trivial reasons. Let us prove the induction step for $n+1$. If $\alpha_{n+1}>\max \{\alpha_0,\ldots,\alpha_n\}$, then both the claims follow immediately from the definition of $e$ and the induction hypothesis. Now let $k,l\le n$ be such that $\alpha_k<\alpha_l$ are neighbours in $\{\alpha_0,\ldots,\alpha_n\}$ and $\alpha_{n+1}$ lies in the interval $(\alpha_k,\alpha_l)$. By the definition of $e$ we have $e(\alpha_{n+1})=e(\alpha_k)+3^{-k}+3^{-n}$. Hence we have $$e(\alpha_{n+1})-e(\alpha_k)=3^{-k}+3^{-n}=3^{-k}+2\cdot 3^{-(n+1)}+3^{-(n+1)}\text{ and}$$ $$\begin{aligned}e(\alpha_l)-e(\alpha_{n+1})& =(e(\alpha_{l})-e(\alpha_k))-(e(\alpha_{n+1})-e(\alpha_k))\\ & \ge 3^{-k}+2\cdot 3^{-n}+3^{-l} -(3^{-k}+3^{-n})\\ &= 3^{-n}+3^{-l}\\ &=3^{-(n+1)}+2\cdot 3^{-(n+1)}+3^{-l}\end{aligned}$$ 
Of course this two (in)equalities imply that $e(\alpha_k)<e(\alpha_{n+1})<e(\alpha_l)$ and the whole claim 1. Also they give us the claim 2. for neighbouring ordinals $\alpha_k,\alpha_{n+1}$ and $\alpha_{n+1},\alpha_l$. Which finishes induction proof, since for all the other neighbouring ordinals the claim 2. follows immediately from induction hypothesis. 

Hence $e$ is order preserving. And for any natural $i$, the $3^{-i}$ neighbourhood of $e(\alpha_i)$ contains no values $e(\beta)$, $\beta<\Lambda$ other than $e(\alpha_i)$ itself. Thus $\sup\limits_{\alpha<\lambda}e(\alpha)<e(\lambda)$, for any limit ordinal $\lambda<\Lambda$.
\end{proof}

The next lemma allows us to extend functions defined on well-ordered sets of reals to continuous functions.

\begin{lemma}\label{Continous_ext}
	Let $ W$ be a well-ordered subset of the real line endowed with its natural order. Assume that $ W$ does not contain any of its limit points. Let $g_0\colon  W^n\to\mathbb{R}$ be a nondecreasing function bounded on precompact sets. Then there exists a continuous nondecreasing function $g:\bbR^n\to\bbR$ such that $g|_{ W^n}=g_0$. 
\end{lemma}

\begin{proof}
	It is convenient to prove a slightly more general statement: Namely, let us make the weaker assumption that $ W$ is well-ordered  in restriction to every half-line of the form $[a, +\infty)$.
	Without losing generality, we may now assume that $ W$ contains a point in every unit interval: Indeed, if the given set
	$ W'$  does not satisfy this property, then we can let 
	$$
	 W= W' \cup \{n\in \mathbb Z: [n-1/2, n+1/2]\cap  W'=\emptyset\}. 
	$$
	If $ W'$ is well-ordered  in restriction to every half-line of the form $[a, +\infty)$, then so is $ W$.
	The function $g_0$ is extended to the additional points by setting, for $\vec m \in W^n\setminus (W')^n$, the value $g_0(\vec m)$  to coincide with the value of $g_0$ at the smallest element of $ (W')^n$ exceeding $\vec m$ coordinate-wise; if such an element does not exist, then the function $g_0$ must be bounded, and we set $g_0(\vec m)$ to be the supremum of its values.
	
	Let $\overline{ W}$ be the closure of $ W$. The set $\overline{ W}$ is countable. We extend the function $g_0$ onto $\overline{ W}^n$ by continuity: for $q\in\overline{ W}^n$ let $q_r\in W^n$ be an increasing sequence converging to $q$ and set $g(q)=\lim_{r\to\infty} g_0(q_r)$. We start by observing that the monotonicity of $g_0$ and the absence of limit points of $ W$ in $ W$ together imply that the resulting extension to $\overline{ W}^n$ is continuous.  Indeed, let $q\in \overline{ W}^n$ be a limit point. One directly checks that a sufficiently small neighbourhood of $q$ only contains points from ${ W}^n$ that are  comparable to $q$ and in fact smaller than $q$. Choosing an arbitrary $\varepsilon>0$ and a point $q_0\in  W^n$
	such that $g(q)<g_0(q_0)+\varepsilon$, we note that in a sufficiently small neighbourhood of $q$ all points from ${ W}^n$ will be strictly greater than $q_0$, and the desired continuity follows.

	Next, let
	$p=(p_1,\dots, p_n)\in\mathbb{R}^n\setminus \overline{ W}^n$. Write
	\begin{align*}
		p_0(i)&=\sup\{t\in W: t\le p_i\},\\
		p_1(i)&=\inf\{t\in W: t\ge p_i\},	\quad i=1,\dots, n.
	\end{align*}
	Then $p_i=\alpha_0(i)p_0(i)+\alpha_1(i)p_1(i)$, where $\alpha_0(i),\alpha_1(i)\in [0,1]$, $\alpha_0(i)+\alpha_1(i)=1$, $i=1,\dots, n$. (If $p_i\in \overline{W}$ then $p_0(i) = p_1(i) = p_i$, so $\alpha_0(i), \alpha_1(i)$ are not uniquely defined. In that case we can let, say, $\alpha_0(i)=\alpha_1(i)=1/2$.)
	For each $\vec\epsilon=(\epsilon_1,\dots,\epsilon_n)\in\{0,1\}^n$, set
	\begin{equation*}
		\alpha_{\vec{\epsilon}}=\alpha_{\epsilon_1}(1)\cdots\alpha_{\epsilon_n}(n),\quad
		p_{\vec{\epsilon}}=(p_{\epsilon_1}(1),\dots,p_{\epsilon_n}(n)).
	\end{equation*}
	Then we have $\sum_{\vec{\epsilon}\in\{0,1\}^n} \alpha_{\vec{\epsilon}} = 1$, and
	\begin{equation*}
		p=\sum_{\vec{\epsilon}\in \{0,1\}^n}\alpha_{\vec{\epsilon}}p_{\vec{\epsilon}}.
	\end{equation*}
	Set
	\begin{equation*}
		g(p)=\sum_{\vec{\epsilon}\in \{0,1\}^n}\alpha_{\vec{\epsilon}}g_0(p_{\vec{\epsilon}}).
	\end{equation*}	
	Since $g|_{{\overline W}^n} = g_0$, the desired extension is constructed.
	
	We now need to check the continuity of the function $g$.
	
	Recall that a uniformly continuous function $f$ on a metric space $(E,d)$ has a modulus of continuity $\omega$ if
	\begin{equation*}
		|f(t_1)-f(t_2)|\le \omega(d(t_1,t_2))
	\end{equation*}
	whenever $d(t_1,t_2)\le 1$.
	For our purposes, it is convenient to metrize $\mathbb R^n$ by $d(\vec t(1),\vec t(2))=\max_i |t_i(1)-t_i(2)|$. In this case, a linear function on an axis-parallel box $C$ of diameter not more than~$1$ has modulus of continuity $\omega(t)=At$ where $A$ is the maximum of the difference of values of our function in the vertices of our box.

	We proceed with our argument. Let $B\subset \mathbb R^n$ be a compact set, so the restriction of $g_0$ to $B$ is uniformly continuous. It suffices to check the continuity of $g$ in restriction to $B$. Let $\omega$ be the modulus of continuity of $g_0|_B$. Assume for simplicity that $B$ is an axis-parallel box with vertices in $\overline {W}^n$.
	
	An axis-parallel box $C\subset \mathbb R^n$ will be called {\it admissible} if it is of the form $C=[p_1, q_1]\times\cdots\times [p_n,q_n]$ where, for each $1\le i\le n$, $p_i, q_i$ are adjacent elements of $\overline W$. Informally speaking, our aim now is to extend the modulus of continuity first to the interior of admissible boxes and then to the whole of $B$.
	
	We have the following immediate
	\begin{lemma}
		Let $t(1), t(2)$ be such that their coordinates either coincide or belong to $\overline  W$. Assume 
		that $d(t(1),t(2))\le 1$
		Then $$
		|g(t(1))-g(t(2))|\le \omega(d(t(1),t(2))).
		$$
	\end{lemma}
	
	\begin{proof} Indeed, by definition we have 	
	$$
	t(1)=\sum_{\vec{\epsilon}\in \{0,1\}^n}\alpha_{\vec{\epsilon}}t_{\vec{\epsilon}}(1);\qquad
	t(2)=\sum_{\vec{\epsilon}\in \{0,1\}^n}\alpha_{\vec{\epsilon}}t_{\vec{\epsilon}}(2),
	$$
	where $t_{\vec{\epsilon}(1)}, t_{\vec{\epsilon}(2)} \in \overline  W^n$: Note that the coefficients   
	$\alpha_{\vec{\epsilon}}$ are the same for $t(1), t(2)$ and also that $d(t(1), t(2))=d(t_{\vec{\epsilon}}(1), 
	t_{\vec{\epsilon}}(2))$ for all $\vec{\epsilon}\in \{0,1\}^n$. Since by definition
	$$
	g(t(1))=\sum_{\vec{\epsilon}\in \{0,1\}^n}\alpha_{\vec{\epsilon}}g(t_{\vec{\epsilon}}(1));\qquad
	g(t(2))=\sum_{\vec{\epsilon}\in \{0,1\}^n}\alpha_{\vec{\epsilon}}g(t_{\vec{\epsilon}}(2)),
	$$
	the statement follows. \end{proof}
	
	Next, in restriction to an admissible box of diameter $\rho\leq 1$, our function has modulus of continuity at most $\omega_1(s)=s\omega(\rho)/\rho$, $s\leq \rho$. Hence, let
	\begin{equation}\label{eq_tilde_omega}
		\tilde{\omega}(t)=t\cdot \max_{\rho\in [t,1]}\Bigl(\frac{\omega(\rho)}{\rho}\Bigr).	
	\end{equation}
	We claim that $\tilde{\omega}(t)\to 0$ as $t\to 0$. Indeed, let $\rho(t)$ be the value of $\rho$ that gives the maximum for a given $t$ in (\ref{eq_tilde_omega}), and note that $\rho(t)$ is nonincreasing in $t$. If $\rho(t)\to 0$ as $t\to 0$, then the fact that $\tilde{\omega}(t)\le \omega(\rho(t))$ implies that $\tilde{\omega}(t) \to 0$. Otherwise, let $c>0$ be the limit of $\rho(t)$ as $t\to 0$. Then $\tilde{\omega}(t)\le t\omega(1)/c$, which again tends to zero with $t$.
	
	Hence, $\tilde{\omega}$ can be chosen as the modulus of continuity of our function restricted to each admissible box within $B$. Furthermore, $\tilde{\omega}(t)\ge \omega(t)$. 
	
	Finally, let $ t(1), t(2)\in B$ be two arbitrary points.
	Consider each coordinate $i=1,\dots, n$ in turn. Say that $t_i(1)<t_i(2)$, and consider the closed interval $I_i=[t_i(1),t_i(2)]$. If $I_i$ contains points of $\overline W$ then let $u_i(1), u_i(2)$ be the minimum and maximum elements of $I_i\cap\overline{W}$, respectively. Otherwise, let $u_i(1)=u_i(2)=t_i(1)$. By construction, $u(1)$ belongs to the same admissible box as $t(1)$, and $u(2)$ belongs to the same admissible box as $t(2)$. Furthermore, none of the distances
	$$d(t(1), u(1)), d(u(1), u(2)), d(u(2), t(2))$$ exceed $d(t(1), t(2))$.
	We consequently have $|g(t(1))-g(t(2))|\leq 3\tilde\omega(d(t(1), t(2)))$, and the desired continuity is established.
\end{proof}

Now we finish the proof of Theorem \ref{thm_cont}. Consider $e$ provided by Lemma \ref{Lambda_emb}. Consider $W=e[\Lambda]$. Clealry $W$ contains none of its limit points. Consider the unique $g_0\colon W^n\to W$ such that $g_0(e(\alpha_1),\ldots,e(\alpha_n))=e(\bar\varphi^{\star}(\alpha_1,\ldots,\alpha_n))$. Applying Lemma \ref{Continous_ext} we get a function $g\colon \mathbb{R}^n\to \mathbb{R}$ such that its restriction to $W$ coincides with $g_0$. From construction it is obvious that $\cF(\{g\},\{0\})$ coincides with $W$ and hence has the order type $\varphi(1,\underbrace{0,\ldots,0}\limits_{\mbox{\footnotesize $n$ times}})$.

\begin{remark} As we already pointed out before $\ot(\cF(G,P))\le o(T(P\sqcup G_0,G_1,\ldots,\allowbreak G_n))$.  Schmidt \cite{Sch79,Sch20}  gives certain lower bounds $\Lambda$ for $o(T(W_1,\ldots,W_n))$. For this she provides a family of monotone functions 
$$\langle f_{i,\alpha}\colon \underbrace{\Lambda\times\ldots\times \Lambda}\limits_{\mbox{\scriptsize $i$ times}}\to\Lambda \mid 0\le i\le n \text{ and } \alpha<o(W_i)\rangle,$$
where $f_{i,\alpha}(\vec{\beta})\le f_{i,\alpha'}(\vec{\beta})$ for any number $0\le i\le n$, ordinals $\alpha<\alpha'<o(W_i)$, and $\vec{\beta}\in \underbrace{\Lambda\times\ldots\times \Lambda}\limits_{\mbox{\scriptsize $i$ times}}$. The family additionally has the property that any ordinal $<\Lambda$ is a value of a closed term built from functions from this family. Although we haven't checked the details, it looks like that in fashion of our proof of Theorem \ref{thm_cont} one could start with this functions and then obtain monotone continuous functions on reals. In particular this would be sufficient to construct the following examples. Suppose we are give natural numbers $k_0,\ldots,k_n$ such that for at least one $i\ge 2$ we have $k_i>0$. Then there are sets $F_1,\ldots,F_n$ of continuous monotone functions on reals, where each $F_i$ is a set of $i$-ary functions and $|F_i|=k_i$ such that
$$\mathsf{ot}(\cF_{F_0\cup\ldots\cup F_n})=o(T(F_0,\ldots,F_n)).$$\end{remark}

\section{Generating the closure of $\cF(G,P)$}\label{sec_closure}
In this section we study the closure of the set $\cF(G,P)$ in the case when $P$ is well-ordered and closed and $G$ consists of finitely many linear functions. We use this analysis later in Section \ref{sec_alg}.

Given finite sets $G$ and $P$, the set $\cF(G,P)$ is not always closed. For example, the set $\cF_3 = \cF(\{g_3\},\{0\})$ does not contain $1/2$, since every element of $\cF_3$ is a fraction whose denominator in reduced form is a power of $3$. Nevertheless, $\cF_3$ contains the numbers $a_n = (1-3^{-n})/2$ for $n\in\bbN$, since $g_3(0,0,a_n)=a_{n+1}$, and these numbers converge to $1/2$.

As we will show, for every $n$ we have $\overline{\cF_n}=\cF_{\leq n} = \cF(\{g_1, \ldots, g_n\},\{0\})$, which, together with Proposition \ref{prop_remove_limits}, implies $\ot(\cF_{\leq n})=\varphi_{n-1}(0)$. As we will see, this is due to the fact that for every $m\le n$ the value $g_m(a_1, \ldots, a_m)$ is equal to the unique $x$ that satisfies the equation $x = g_n(a_1, \ldots, a_m, x, \ldots, x)$. We now develop this idea in more generality.

Given a monotone unary linear function $h(x)=ax+b$, $0\le a<1$, let $\mathsf{fix}(h)=b/(1-a)$ be the fixed point of $h$. Note that for every $x<\mathsf{fix}(h)$, the sequence $x,h(x), h(h(x)), h^{(3)}(x), \ldots$ is increasing and converges to $\mathsf{fix}(h)$.

Recall the definition of $\dot g_{u,\vec{p}}(x)$ from Section \ref{Upper_bound_section}.
Let us define the set of functions $\closure(g)$ for a linear function $g$. The set $\closure(g)$ consists of certain functions $g_u\colon \mathbb{R}^{n-|u|}\to \mathbb{R}$, for some subsets $u\subseteq \{1,\ldots,n\}$. Consider some $u\subseteq \{1,\ldots,n\}$. If the unary linear function $\dot g_{u,\vec{x}}$ is of the form $ax+b$, where $a\ge 1$, then we do not define $g_u$ and don't have it in $\closure(g)$. Otherwise we put $g_u(\vec{y})=\mathsf{fix}(\dot g_{u,\vec{y}})$ and add it to $\closure(g)$.

In particular, $g_{\{1,\ldots,n\}}$, if defined, is the constant function that takes no inputs and returns the unique $x$ that satisfies $g(x,\ldots,x)=x$. At the other extreme, $g_\emptyset$ is $g$ itself, hence $g\in\closure(g)$.

For a finite set of linear functions $G$ we put $\closure(G)=\bigcup\limits_{g\in G}\closure(g)$.

\begin{proposition} \label{prop_closure}Suppose $G$ is a finite set of linear functions and $P$ is a well-ordered and closed set of constants such that for any $g\in G$ there is $p_g\in P$ such that $g(p_g,\ldots,p_g)>p_g$. Then $\overline{\cF(G,P)}=\cF(\closure(G),P)$.\end{proposition}
\begin{proof}
  First let us prove that $\cF(\closure(G),P)\subseteq \overline{\cF(G,P)}$. This direction is done by ordinary induction. Let $a\in \cF(\closure(G),P)$. If $a\in P$, then there is nothing to prove. Further assume that $a$ is of the form $f(\vec{p})$, for some $f\in \closure(G)$, $\vec{p}=(p_1,\ldots,p_k)$, $\max(\vec{p})<a$, $\vec{p}\in \cF(\closure(G),P)^k$. By the induction assumption we have $\vec{p}\in \bigl(\overline{\cF(G,P)}\bigr)^k$, so we can choose an infinite coordinate-wise non-decreasing sequence $\vec{s}_{0},\vec{s}_{1},\ldots$ such that each  $\vec{s}_i=(s_{i,1},\ldots,s_{i,k})\in (\cF(G,P))^k$, and such that the limit of the sequence is $\vec{p}$. We now define a sequence $q_{0},q_1,\ldots$ of elements of $\cF(G,P)$ that converges to $a$. If $f = g_\emptyset=g$ for some $g\in G$, then we put $q_i=g(\vec{s}_i)$, for all $i$. Now suppose $f$ is $g_{u}$, for some $n$-ary $g\in G$ and nonempty $u\subseteq \{1,\ldots,n\}$. If $u=\{1,\ldots,n\}$, then we put $q_0=p_g$ and $q_{i+1}=g(q_i,\ldots,q_i)$. Hence, suppose $u\subsetneq \{1,\ldots,n\}$, so $k\ge 1$. Let $p_\mathrm{max} = \max(\vec{p})$. We have $p_\mathrm{max} < \dot g_{u,\vec{p}}(p_\mathrm{max}) < a$. If $\vec{s}_{0}$ is close enough coordinate-wise to $\vec{p}$ (which can be assumed without loss of generality), then we still have $\dot g_{u,\vec{s}_0}(\max(\vec{s}_0))>p_\mathrm{max}$. Hence, in this case we put $q_0=\max(\vec{s}_0)$ and $q_{i+1}=\dot g_{u,\vec{s}_i}(q_{i})$.  In all cases, it is easy to see that $q_i$'s indeed are elements of $\cF(G,P)$ and that $\lim_{i\to \infty} q_i= a$. Hence, $a\in\overline{\cF(G,P)}$, as desired.
  
  Since $\cF(G,P)$ is well-ordered, the set $\overline{\cF(G,P)}$ is also well-ordered. By transfinite induction on $a\in\overline{\cF(G,P)}$ we now prove that $a\in\cF(\closure(G),P)$. If $a$ is not a limit of an infinite increasing sequence of elements of $\cF(G,P)$, then $a\in \cF(G,P)\subseteq  \cF(\closure(G),P)$. So further we assume that $a$ is a limit of an infinite increasing sequence $s_0<s_1<\ldots$ of elements of $\cF(G,P)$. By switching to a subsequence if necessary, we can assume without loss of generality that either (1) all $s_i$ are in $P$, or (2) for some fixed $g\in G$ and all $i$ the numbers $s_i$ are of the form $g(\vec{p}_i)$, where $\vec{p}_i=(p_{i,1},\ldots,p_{i,n})\in (\cF(G,P))^n$. In case (1), $a=\lim_{i\to\infty}s_i\in P$ since $P$ is closed and hence  $a\in\cF(G,P)$. Now consider case (2). By applying the infinite Ramsey's Theorem if necessary, we can assume that for any $1\le j\le n$ the sequence $p_{0,j},p_{1,j},\ldots$ is either strictly increasing, or constant, or strictly decreasing. Since $\cF(G,P)$ is well-ordered, these sequences cannot be strictly decreasing. Let $u$ be the set of all indexes $1\le j\le n$ such that $\lim_{i\to\infty} p_{i,j}=r_j<a$. Let $\vec{r}$ be the vector of all $r_j$'s, for $j\in u$. By transfinite induction assumption $\vec{r}\in \cF(\closure(G),P)^{|u|}$. By continuity of $g$ we have $$a=g(\lim_{i\to \infty}p_{i,1},\ldots,\lim_{i\to \infty}p_{i,n})=\dot g_{u,\vec{r}}(a)=g_u(\vec{r})\in \cF(\closure(G),P).$$
\end{proof}

\section{Computational aspects}\label{sec_alg}

In this section we prove Theorem \ref{thm_alg}. Recall Remark \ref{R_algorithms} on the models of computation over the reals and the rationals. Let $P$ be a finite set of initial elements, let $g:\bbR^n\to\bbR$ be a monotone linear function, and let $\cF=\cF(\{g\},P)$. Table \ref{tab_algorithms} shows four interdependent procedures:
\begin{itemize}
\item $\Succ(r)$ returns the smallest element of $\cF$ larger than $r$.
\item $\BuiltSucc(r, (y_1, \ldots, y_k))$, when given $y_1,\ldots,y_k\le r$, returns the smallest number $s>r$ of the form $s=g(y_1, \ldots, y_k, x_{k+1}, \ldots, x_n)$ for $x_{k+1}, \ldots, x_n\in\cF\cap(-\infty,r]$.
\item $\Pred(r)$ returns the predecessor in $\overline{\cF}$ of a given successor element $r\in\overline{\cF}$. (If $r$ is not a successor element in $\overline{\cF}$ then $\Pred$ runs forever.) Recall that, by Proposition \ref{prop_closure}, we have $\overline{\cF} = \cF(\closure(g),P)$.
\item $\WeakPred(r)$ is similar to $\Pred$, and returns the unique element $x\in\overline{\cF}$ such that $x\le r$ and the successor of $x$ is larger than $r$. (If $r<\min P$ then $\WeakPred$ returns $-\infty$.) Unlike $\Pred$, $\WeakPred$ halts on all inputs.
\end{itemize}
These four procedures call one another recursively (see Figure \ref{fig_algs}).

\begin{table}
    \centering
    \begin{algorithm}[H]
\caption{}
\begin{algorithmic}[5]
\State\LeftComment{Here $g(x_1, \ldots, x_n) = a_1 x_1 + \cdots + a_n x_n + b$}
\Procedure{Succ}{$r$}
\State $m\gets$\Call{BuiltSucc}{$r$,$()$}
\State $m'\gets \min{(P\cap(r,\infty))}$
\State\Return{$\min{\{m,m'\}}$}
\EndProcedure
\Procedure{BuiltSucc}{$r$, $(y_1, \ldots, y_k)$}
\If{$k=n$}
    \State\Return{$g(y_1, \ldots,y_n)$}
\EndIf
\State $x_{k+1}\gets$\Call{WeakPred}{$r$}
\State $m\gets\infty$
\While{$g(y_1,\ldots,y_k,x_{k+1},r,\ldots,r)> r$}
    \State $s\gets$\Call{BuiltSucc}{$r$, $(y_1, \ldots, y_k, x_{k+1})$}
    \State $x'\gets x_{k+1}-(s-r)/a_{k+1}$
    \State $x''\gets$\Call{Succ}{$x'$}
    \State $m'\gets r+a_{k+1}(x''-x')$
    \If{$m'<m$}
        \State $m\gets m'$
    \EndIf
    \State $x_{k+1}\gets$\Call{Pred}{$x''$}
\EndWhile
\State\Return{$m$}
\EndProcedure
\Procedure{Pred}{$r$}
\LeftComment{Let $z_1, z_2, z_3, \ldots$ be an enumeration of the elements of $\overline{\cF}$}
\If{$r\le\min{P}$}
    \State\Return{$-\infty$}
\EndIf
\For{$i=1,2,3,\ldots$}
    \If{$z_i<r$ and \Call{Succ}{$z_i$}${}=r$}
        \State\Return{$z_i$}
    \EndIf
\EndFor
\EndProcedure
\Procedure{WeakPred}{$r$}
\LeftComment{Let $z_1, z_2, z_3, \ldots$ be an enumeration of the elements of $\overline{\cF}$}
\If{$r<\min{P}$}
    \State\Return{$-\infty$}
\EndIf
\For{$i=1,2,3,\ldots$}
    \If{$z_i=r$ or ($z_i<r$ and \Call{Succ}{$z_i$}${}> r$)}
        \State\Return{$z_i$}
    \EndIf
\EndFor
\EndProcedure
\end{algorithmic}
\end{algorithm}
    \caption{Inter-dependent recursive algorithms for computational problems related to $\cF(\{g\},P)$ for a linear function $g$.}
    \label{tab_algorithms}
\end{table}

\begin{figure}
    \centering
    \includegraphics{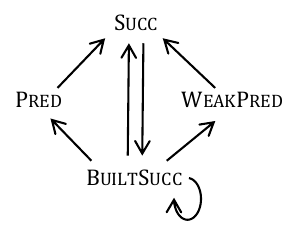}
    \caption{Interdependence between the procedures of Table \ref{tab_algorithms}.}
    \label{fig_algs}
\end{figure}

Procedures $\Pred$ and $\WeakPred$ rely on the fact that the set $\overline{\cF}=\cF(\closure(g),P)$ is computably enumerable: For example, one can list its elements by increasing number of function applications needed to construct them, breaking ties somehow.

We now prove correctness of all four procedures. The following proposition is self-evident:

\begin{proposition}
Procedures $\Succ$ and $\WeakPred$ terminate and return the correct value under the assumption that all the procedure calls made by them terminate and return the correct values. For a successor element $r\in\overline{\cF}$, procedure $\Pred(r)$ terminates and returns the correct value under the assumption that all its procedure calls terminate and return the correct values.
\end{proposition}

\begin{lemma}
Procedure $\BuiltSucc$ terminates and returns the correct value under the assumption that all its procedure calls terminate and return the correct values.
\end{lemma}

\begin{proof}
Procedure $\BuiltSucc$ searches for the optimal value of $x_{k+1}$ in decreasing order, by starting with the weak predecessor of $r$, and stopping when $x_{k+1}$ is too small to yield values larger than $r$. At the beginning of the \texttt{while} loop on line 11, the value of $x_{k+1}$ always belongs to $\overline{\cF}$, and $m$ always holds the smallest value of the form $s=g(y_1, \ldots, y_k, x^\star, x_{k+2}, \ldots, x_n)$ with $s>r$ and $x^\star\in\cF\cap(x_{k+1},r]$ and $x_{k+2}, \ldots x_n\in\cF\cap(-\infty,r]$.

The value of $s$ on line 12 equals $g(y_1, \ldots, y_k, x_{k+1}, x_{k+2}, \ldots, x_n)$ for some optimal choice of $\bar x=(x_{k+2}, \ldots, x_n)$. This choice of $\bar x$ is also optimal if we replace $x_{k+1}$ by any $x^\star\in\cF\cap(x', x_{k+1})$ for the value of $x'$ computed on line 13, since changing $x_{k+1}$ by an amount $u$ changes the output of $g$ by $a_{k+1} u$. Hence, the best value to give $x_{k+1}$ under the choice of $\bar x$ is $x''$ as computed on line 14. Note that $x''$ will never be larger than $x_{k+1}$.

The \texttt{while} loop will always terminate after a finite number of iterations, since at the top of the loop $x_{k+1}$ always belongs to $\overline{\cF}$ and decreases from one iteration to the next.
\end{proof}

\begin{lemma}
Procedures $\Succ$, $\WeakPred$, and $\BuiltSucc$ terminate and return the correct value. For a successor element $r\in\overline{\cF}$, procedure $\Pred(r)$ terminates and returns the correct value.
\end{lemma}

\begin{proof}
All we need to do is rule out the possibility of an infinite chain of procedure calls. So suppose for a contradiction that $P_1, P_2, P_3, \ldots$ is an infinite chain of procedure calls, where each $P_i$ is one of our four procedures, and each $P_i$ was called with $r=r_i$ and it in turn called $P_{i+1}$. We have $r_1\ge r_2\ge r_3\ge \cdots$. For each $i\in\bbN$ let $\sigma_i=\min(\overline{\cF}\cap[r_i,\infty))$. Hence, $\sigma_1\ge\sigma_2\ge\sigma_3\ge\cdots$. Since $\overline{\cF}$ is well ordered, there exist only finitely many indices $i$ for which $\sigma_i>\sigma_{i+1}$. Let $i^\star$ be the index at which $\sigma_i$ reaches its minimum value $\sigma^\star$. Whenever $P_i$ is $\Pred$ or $\WeakPred$ we have $\sigma_i>\sigma_{i+1}$. Hence, for all $i\ge i^\star$, $P_i$ is either $\Succ$ or $\BuiltSucc$. 

Since $\BuiltSucc$, when given a sequence $(y_1, \ldots, y_k)$ of length $k$, only calls itself on sequences of length $k+1$, we cannot have $P_i=\BuiltSucc$ more than $n$ times in a row. Hence, there exist infinitely many indices $i\ge i^\star$ for which $P_{i}=\BuiltSucc$ and $P_{i+1}=\Succ$, in which case $\BuiltSucc$ called $\Succ$ on line $14$. For these indices $i$, the value of $s$ calculated on line $13$ is at least as large as $\sigma^\star$. Hence, the distance between $r_{i+1}$ and $s^\star$ is larger than the distance between $r_{i}$ and $s^\star$ by a factor of at least $(1+ 1/\max\{a_1, \ldots, a_n\})$. Hence, $\lim_{i\to\infty} r_i=-\infty$. However, when $\BuiltSucc$ is given $r<\min P$, it sets $x_{k+1}$ to $-\infty$ on line $9$, and then the condition of the \texttt{while} loop on line $11$ does not hold, so no call is made to $\Succ$. Contradiction.
\end{proof}

Theorem \ref{thm_alg} now follows, since $r\in\overline{\cF}$ if and only if $\WeakPred(r)=r$. As mentioned, we do not know whether there exists an algorithm that decides whether $r\in\cF$ or not.

\subsubsection*{Acknowledgements} Thanks to the anonymous referee for carefully reading the paper and providing helpful comments.




\bibliographystyle{alpha}
\bibliography{bibliography}

\end{document}